\documentclass{amsart}

\usepackage{graphicx}
\usepackage[
	pdftitle={},
	pdfauthor={},
	ocgcolorlinks,
	linkcolor=linkblue,
	citecolor=linkred,
	urlcolor=linkred]
{hyperref}
\usepackage{color}
\definecolor{linkred}{rgb}{0.6,0,0}
\definecolor{linkblue}{rgb}{0,0,0.6}
\definecolor{orange}{rgb}{1,0.6,0}

\usepackage{amsmath, amsfonts, amssymb,
amscd,mathtools
}

\usepackage{eulervm}

\usepackage{eulervm}

\newtheorem{thm}{Theorem}
\newtheorem*{repeatthm}{Theorem~\ref{th:systole}}
\newtheorem*{repeatthm2}{Theorem~\ref{th:main}}
\newtheorem*{repeatthm3}{Theorem~\ref{th:asymp}}
\newtheorem{dfn}{Definition}	

\theoremstyle{remark}
\newtheorem{remark}{Remark}

\theoremstyle{plain}
\newtheorem{lem}[thm]{Lemma}
\newtheorem{prop}[thm]{Proposition}
\newtheorem{cor}[thm]{Corollary}

\providecommand{\arcsinh}{\mathop{\rm arcsinh}\nolimits}
\providecommand{\arccosh}{\mathop{\rm arccosh}\nolimits}
\providecommand{\Card}{\mathop{\rm Card}\nolimits}
\providecommand{\tr}{\mathop{\rm tr}\nolimits}
\providecommand{\GL}{\mathop{\rm GL}\nolimits}
\providecommand{\SL}{\mathop{\rm SL}\nolimits}
\providecommand{\PGL}{\mathop{\rm PGL}\nolimits}
\providecommand{\PSL}{\mathop{\rm PSL}\nolimits}
\providecommand{\sys}{\mathop{\rm sys}\nolimits}
\providecommand{\Vol}{\mathop{\rm Vol}\nolimits}
\providecommand{\Isom}{\mathop{\rm Isom}\nolimits}
\providecommand{\Sim}{\mathop{\rm Sim}\nolimits}
\providecommand{\sym}{\mathop{\rm sym}\nolimits}

\newcommand{\cev}[1]{\reflectbox{\ensuremath{\vec{\reflectbox{\ensuremath{#1}}}}}}

\title{Simple geodesics and Markoff quads} 

\author{Yi Huang and Paul Norbury}
\address{Department of Mathematics and Statistics, The University of Melbourne, Australia 3010.}
\email{\href{mailto:y.huang@ms.unimelb.edu.au}{y.huang@ms.unimelb.edu.au}, \href{mailto:norbury@unimelb.edu.au}{norbury@unimelb.edu.au}}
\subjclass[2010]{32G15, 58D27, 30F60}
\thanks{This work was supported by the Australian Research Council grant DP1094328.}

\begin{document}
\maketitle

\begin{abstract}
The action of the mapping class group of the thrice-punctured projective plane on its $\GL(2,\mathbb{C})$ character variety produces an algorithm for generating the simple length spectra of quasi-Fuchsian thrice-punctured projective planes.  We apply this algorithm to quasi-Fuchsian representations of the corresponding fundamental group to prove: a sharp upper-bound for the length of its shortest geodesic, a McShane identity and the surprising result of non-polynomial growth for the number of simple closed geodesic lengths. 
\end{abstract}

\section{Introduction}

\subsection{Background}

Closed geodesics on hyperbolic surfaces have extremely rich properties, arising in geometry, topology and number theory, and this is particularly true of \emph{simple} (that is: non-self-intersecting) closed geodesics. The central objects of our study are simple closed geodesics on 3-cusped projective planes. We describe an algorithm for generating these geodesics and their lengths, and use this algorithm to study systoles, McShane identities and simple length spectra of 3-cusped projective planes.

\begin{figure}[ht]
\begin{center}
\includegraphics[scale=.4]{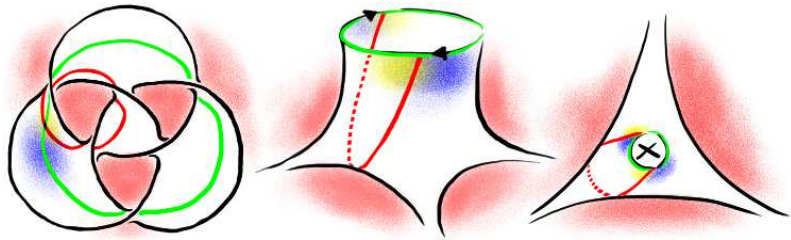}	
\end{center}
\caption{Simple curves on thrice-punctured projective planes.}
\label{fig:borromean}
\end{figure}

\textbf{Systoles:} The \emph{systole} $\sys(M)$ of a Riemannian manifold $M$ is the length of its shortest essential closed curve which is necessarily a simple closed geodesic. First mentioned in \cite{PuPSom}, Loewner's torus inequality gives the following relationship between the systole $\sys(T)$ of a Riemannian torus $T$ and its surface area $\Vol(T)$:
\begin{align*}
\sys(T)^2\leq\tfrac{2}{\sqrt{3}}\,\Vol(T),
\end{align*}
equality is realised for tori obtained by gluing opposite sides of regular hexagons. Since then, systolic inequalities have been obtained for other surfaces and higher dimensional manifolds such as Gromov's inequality \cite{GroFil} which holds for a large class of Riemannian manifolds. \newline
\newline
More recently, there has been growing interest in systolic hyperbolic geometry. For example, for orientable hyperbolic surfaces $X$ with $n\geq2$ cusps and/or boundary components \cite{SchCon}: 
\begin{align}
\sys(X)\leq 4\arccosh\left(\frac{3|\chi(X)|}{n}\right).
\end{align}
Hyperbolic systolic inequalities do not require a volume term on the right, because the area of a hyperbolic surface $X$ is topologically determined: it is $2\pi|\chi(X)|$.\newline
\newline
\textbf{McShane identities:} A \emph{McShane identity} may be thought of as a type of trigonometric identity for hyperbolic surfaces. It is usually a sum over functions of lengths of simple closed geodesics on a punctured/bordered hyperbolic surface. In particular, the structure of this sum is independent of the hyperbolic structure on the surface.\newline
\newline
McShane proved the first such identity in his doctoral dissertation \cite{McSRem}. Denote the collection of simple closed geodesics on a 1-cusped torus $X_{1,1}$ by $\mathcal{S}_{1,1}$, then:
\begin{align}
1=\sum_{\gamma\in\mathcal{S}_{1,1}}\frac{2}{1+\exp\ell_\gamma},
\end{align}
where $\ell_\gamma$ denotes the length of the simple closed geodesic $\gamma$. Each term in the above series corresponds to the probability that a geodesic launched from the cusp on $X_{1,1}$, up to its first point of self-intersection, does not intersect $\gamma$. McShane identities for other hyperbolic surfaces and quasi-Fuchsian representations have steadily followed \cite{AkiVar,HuHIde,HuaMcS,LuoDil,McSRem,McSSim,MirSim,NorLen,TWZGen}.\newline
\newline
\textbf{Simple length spectra:} Given a Riemannian surface, its \emph{length spectrum} is the multiset of lengths of closed geodesics on the surface. The Selberg trace formula interprets the length spectrum in terms of the spectrum of the Laplace-Beltrami operator \cite{SelHar}, and can be used to show that the number of closed geodesics of length less than $L$ on an orientable hyperbolic surface grows exponentially.\newline
\newline
The \emph{simple length spectrum} is the multiset of lengths of simple closed geodesics. Simple closed geodesics are relatively rare among closed geodesics. The growth rate of the simple length spectrum on an orientable hyperbolic surface is only polynomial in $L$. This was first proven using combinatorial arguments \cite{MRiSim,RivSim}. Mirzakhani gives a novel proof of a refinement of this result in \cite{MirGro}: let $\mathcal{S}(X)$ denote the set of simple closed geodesics on a hyperbolic surface $X$, for $L>0$ define
\begin{align*}
s_{X}(L)=\Card\left\{\gamma\in \mathcal{S}(X) \mid \ell_{\gamma}(X)<L\right\}.
\end{align*}
Then, the function $\eta:\mathcal{M}(X)\rightarrow\mathbb{R}_+$ defined by taking the limit
\begin{align}  \label{gropol}
\lim_{L\to\infty}\frac{s_X(L)}{ L^{\dim_{\mathbb{R}}\mathcal{M}(X)}}=\eta(X)>0
\end{align}
 is a continuous proper function. Mirzakhani's proof employs the ergodicity of the mapping class group action on the space of geodesic measured laminations, as well as her calculations of moduli spaces volumes \cite{MirSim} --- calculations utilising McShane identities.\newline
\newline
\textbf{Markoff triples:} A \emph{Markoff triple} is a solution $(x,y,z)\in\mathbb{C}^3$ to the equation:
\begin{align}  \label{triple}
x^2+y^2+z^2=xyz.
\end{align}
The hypersurface in $\mathbb{C}^3$ defined by \eqref{triple} is the relative character variety of the fundamental group of the once-punctured torus. In particular, this gives a real analytic diffeomorphism between the set of positive real Markoff triples and the Fuchsian component of the character variety --- a model for the Teichm\"uller space of hyperbolic 1-cusped tori.\newline
\newline
The following transformations:
\begin{align}\label{tripletrans}
(x,y,z)\mapsto
(x,y,xy-z)
\text{ and }(x,y,z)\mapsto(y,z,x),
\end{align}
take one Markoff triple to another. In particular, these transformations \eqref{tripletrans} generate the extended mapping class group of the punctured torus, and describe its action on the corresponding relative character variety. For Fuchsian characters, any triple $(x,y,z)\in\mathbb{R}^3_+$ consists of $2\cosh(\frac{1}{2}\cdot)$ of the lengths of an ordered triple of simple closed geodesics on a particular 1-cusped torus $X$. One is thus able to generate the entire simple length spectrum of $X$ by applying sequences of the transformations in \eqref{tripletrans} to $(x,y,z)$.\newline
\newline 
This relationship between positive real Markoff triples and the simple length spectra of hyperbolic 1-cusped tori was first exploited in \cite{CohApp}. In \cite{BowPro}, Bowditch uses a generalisation of this correspondence to derive a sharp systolic inequality for quasi-Fuchsian representations of 1-cusped tori, and also to establish a quasi-Fuchsian generalisation of McShane's original identity. The length generation algorithm can also be used to prove that quasi-Fuchsian representations of the punctured torus group have $L^2$~simple length growth rates \cite{HuaMod}.\newline

\subsection{Markoff Quads}
\begin{dfn}
\emph{Markoff quads} are 4-tuples $(a,b,c,d)\in\mathbb{C}^4$ of complex numbers satisfying:
\begin{align} \label{quad}
(a+b+c+d)^2=abcd.
\end{align}
\end{dfn}
The hypersurface in $\mathbb{C}^4$ defined by \eqref{quad} is the relative character variety of the fundamental group of the thrice-punctured projective plane $N_{1,3}$. In other words, we have the following bijective correspondence:
\begin{align*}
\left\{\text{ Markoff quads }\right\}\xleftrightarrow{1:1}&\left\{\text{ characters of }\GL(2,\mathbb{C})\text{-representations of }\pi_1(N_{1,3})\right\}.
\end{align*}
Characters of Fuchsian and quasi-Fuchsian representations are of special importance because they respectively arise as monodromy representations of hyperbolic surfaces and hyperbolic 3-manifolds. We now describe this relationship and introduce some notation for the rest of the paper.\newline
\newline
\textbf{Fuchsian:} Given a Fuchsian representation of the fundamental group of a (possibly non-orientable) surface $S$
\begin{align*}
\rho:\pi_1(S)\rightarrow\PSL^{\pm}(2,\mathbb{R})=\Isom^{\pm}(\mathbb{H}),
\end{align*}
the discrete subgroup $\rho(\pi_1(S))$ acts properly discontinuously on the hyperbolic plane $\mathbb{H}$ by (possibly orientation reversing) isometries. The quotient $\mathbb{H}/\rho(\pi_1(S))$ is a complete hyperbolic surface homeomorphic to $S$, and we denote it by $X_\rho$. By identifying the universal cover of $S$ with $\mathbb{H}$, Fuchsian representations induce a homeomorphism $h_\rho:S\rightarrow X_\rho$, canonical up to homotopy.\newline
\newline
\textbf{Quasi-Fuchsian:} Similarly, for a strictly quasi-Fuchsian representation
\begin{align*}
\rho:\pi_1(S)\rightarrow\PSL(2,\mathbb{C})=\Isom^+(\mathbb{H}^3), 
\end{align*}
the discrete subgroup $\rho(\pi_1(S))$ acts properly discontinuously on $\mathbb{H}^3$ by orientat\-ion preserving isometries. The quotient space $\mathbb{H}^3/\rho(\pi_1(S))$ is an orientable complete hyperbolic 3-manifold homeomorphic to $(0,1)\times S$, and we also denote it by $X_\rho$. In analogy to the Fuchsian case, quasi-Fuchsian representations induce a canonical (up to homotopy) embedding $h_\rho: S\hookrightarrow X_\rho$. \newline
\newline
\textbf{Length functions on character varieties:} Given a simple closed curve $\gamma$ on $S$, there is a unique simple closed geodesic on $X_\rho$ homotopy equivalent to $h_\rho(\gamma)$. This allows one to define a function $\ell_{\gamma}(\cdot)$ on the space of quasi-Fuchsian representations of $\pi_1(S)$ by taking the \emph{complex length} of the unique closed geodesic homotopy equivalent to $h_\rho(\gamma)$ in $X_\rho$. The complex length of a geodesic has real and imaginary parts respectively given by its geometric length and the angle of twisting of the normal bundle around the closed geodesic.

\subsection{Main results}
\begin{thm}[Systolic inequality] \label{th:systole}
Let $\rho$ denote a quasi-Fuchsian monodromy representation for a thrice-punctured projective plane, then
\begin{align}
\sys(X_\rho)\leq2\arcsinh(2).
\end{align}
In particular, the unique maximum of the systole function over the moduli space of all hyperbolic thrice-punctured projective planes is $2\arcsinh(2)$.
\end{thm}

\begin{remark}
The unique systolic maximum is realised by the 3-cusped projective plane with the largest isometry subgroup. This symmetric surface is doubly covered by a hyperbolic surface conformally equivalent to the unit sphere in $\mathbb{R}^3$ minus the 6 points where it meets the three axes. Its simple length spectrum can be generated from the integral Markoff quad~$(4,4,4,4)$.
\end{remark}

\begin{thm}[McShane identity]  \label{th:main}
Let $\rho$ be a quasi-Fuchsian representation of the thrice-punctured projective plane fundamental group $\pi_1(N_{1,3})$. Then,
\begin{align*}
\sum_{\gamma\in\Sim_2(N_{1,3})}\frac{1}{1+\exp{\tfrac{1}{2}\ell_{\gamma}(\rho)}}=\frac{1}{2},
\end{align*}
where the sum is over the collection $\Sim_2(N_{1,3})$ of free homotopy classes of essential, non-peripheral two-sided simple closed curves $\gamma$ on $N_{1,3}$.
\end{thm}

\begin{remark}
This result is known for Fuchsian representations owing to the second author's work in \cite{NorLen}. Moreover, Hu, Tan  and Zhang \cite{HuHIde} have derived a McShane identity for solutions of Markoff-Hurwitz equations (see subsection~\ref{markoff-hurwitz}). In the $n=4$ case, their identity coincides with ours for quasi-Fuchsian thrice-punctured projective planes after a coordinate change. Thus, our result affirmatively answers their question of whether their identity has a geometric interpretation in the $n=4$ case.
\end{remark}

Take note that theorem~\ref{th:main} is a series over $\Sim_2(N_{1,3})$, the set of two-sided simple closed curves on $N_{1,3}$, rather than over the collection $\mathcal{S}_2(X_\rho)$ of two-sided simple closed \emph{geodesics} on the hyperbolic 3-manifold $X_\rho$. This is because for non-Fuchsian representations $\rho$, the geodesic representatives of $\Sim_2(N_{1,3})$ constitutes only a subset of $\mathcal{S}_2(X_\rho)$. Similarly, we consider the subset of the simple length spectrum of a quasi-Fuchsian representation $\rho$ corresponding to the set $\Sim_1(N_{1,3})$ of \emph{one-sided} simple closed curves on $N_{1,3}$, and study the growth rate of the following quantity:
\begin{align*}
s_\rho(L)=\Card\left\{
\gamma\in\Sim_1(N_{1,3}) \mid |\ell_{\gamma}(\rho)|<L\right\}.
\end{align*}
When $\rho$ is Fuchsian, the value of $s_\rho(L)$ is equal to the number $s_{X_\rho}(L)$ of geodesics on $X_\rho$ below length $L$.

\begin{thm} \label{th:asymp}
Given a quasi-Fuchsian representation $\rho$ of the thrice-punctured projective plane $N_{1,3}$,
\begin{align*}
\lim_{L\to\infty}\frac{s_\rho(L)}{L^m}>0
\end{align*}
for some $m$ satisfying $2.430<m < 2.477$.
\end{thm}

\subsection*{Acknowledgements:}  The authors are grateful to Craig Hodgson for useful conversations, to Andrew Elvey-Price and Greg McShane for helping us to improve the bounds in Theorem~\ref{th:asymp}, and to Ser Peow Tan and Hengnan Hu for conversations about their work.

\section{Markoff Quads}  \label{sec:mq}

Consider a 4-tuple $(\alpha,\beta,\gamma,\delta)$ of distinct one-sided simple closed curves on a thrice-punctured projective plane $N_{1,3}$ that pairwise intersect once.  Figure~\ref{fig:flip} shows two such 4-tuples $({\color{red}\alpha},{\color{red}\beta},{\color{red}\gamma},{\color{blue}\delta})$ and $({\color{red}\alpha},{\color{red}\beta},{\color{red}\gamma},{\color{blue}\delta'})$; the depicted crossed circle is a cross-cap which represents an embedded M\"obius strip. 
\begin{figure}[ht]
\begin{center}
\includegraphics[scale=.4]{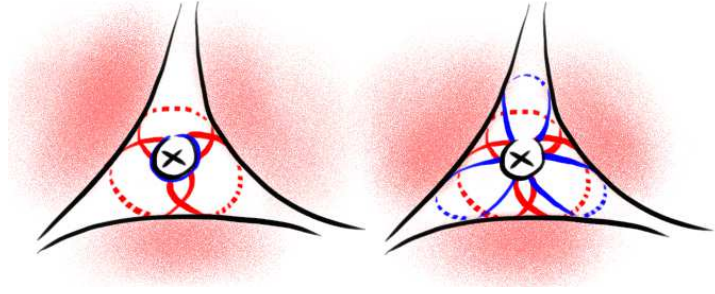}	
\end{center}
\caption{Flipping the blue curve.}
\label{fig:flip}
\end{figure}

Up to homotopy, the curves ${\color{blue}\delta}$ and ${\color{blue}\delta'}$ are the only one-sided simple closed curves that intersect each of the curves ${\color{red}\alpha},{\color{red}\beta},{\color{red}\gamma}$ exactly once. We call the process of replacing ${\color{blue}\delta}$ with ${\color{blue}\delta'}$ and vice versa, a \emph{flip}. For any quasi-Fuchsian representation $\rho:\pi_1(N_{1,3})\rightarrow\PSL(2,\mathbb{C})$, a Fricke trace identity \cite{MagRin} shows that
\begin{align*}
(a,b,c,d)=(2\sinh\tfrac{1}{2}\ell_{\alpha}(\rho),2\sinh\tfrac{1}{2}\ell_{\beta}(\rho),2\sinh\tfrac{1}{2}\ell_{\gamma}(\rho),2\sinh\tfrac{1}{2}\ell_{\delta}(\rho))
\end{align*}
satisfies equation \eqref{quad} and is therefore a \emph{Markoff quad}. That is:
\begin{align*}
(a+b+c+d)^2=abcd.
\end{align*}
Just as we can flip from $(\alpha,\beta,\gamma,\delta)$ to $(\alpha,\beta,\gamma,\delta')$, a new Markoff quad may be obtained via the following transformation:
\begin{align}  \label{flip}
(a,b,c,d)\mapsto(a,b,c,d'=abc-2a-2b-2c-d),
\end{align}
where $d'=2\sinh\frac{1}{2}\ell_{\delta'}(\rho)$. Similarly flip $\alpha$ to $\alpha'$, $\beta$ to $\beta'$ or $\gamma$ to $\gamma'$, to correspondingly obtain three other transformations:
\begin{align*}
(a,b,c,d)\mapsto\arraycolsep=1.4pt\def\arraystretch{1.2}
\begin{array}{l}
(bcd-2b-2c-2d-a,b,c,d),\\ 
(a,acd-2a-2c-2d-b,c,d),\\ 
(a,b,abd-2a-2b-2d-c,d),
\end{array}
\end{align*}
which take $(a,b,c,d)$ to new Markoff quads. Every one-sided simple closed curve can be uniquely obtained by some sequence of flips \cite{SchCom}, and thus \eqref{flip} gives us an algorithm for generating the (one-sided) simple length spectrum of hyperbolic 3-cusped projective planes $X_\rho$.\newline
\newline
We begin this section by considering the trace identities needed for this algorithm, before detailing how to store the combinatorics of Markoff quads (and hence the simple length spectrum) for a 3-cusped projective plane in its curve complex.

\subsection{Trace identities} \label{sec:trid} The fundamental group $\pi_1(N_{1,3})$ for thrice-punctured projective planes is the free froup $F_3$ of rank 3. Any representation $\rho:\pi_1(N_{1,3})\rightarrow \PSL(2,\mathbb{C})$ of a free group admits a lift to a representation
\begin{align*}
\tilde{\rho}:\pi_1(N_{1,3})\rightarrow \SL(2,\mathbb{C}).
\end{align*}
The character $\xi=\tr\tilde{\rho}$ completely determines the geometry of $X_\rho$, and so we consider the $\PSL(2,\mathbb{C})$-character variety of
\begin{align*}
\pi_1(N_{1,3})=F_3=\left\langle\,\alpha_1,\alpha_2,\alpha_3\mid -\; \right\rangle. 
\end{align*}
We set $A_i\in\SL(2,\mathbb{C})$ to denote the matrix $\tilde{\rho}(\alpha_i)$.\newline
\newline
Goldman showed \cite{GolTra} that any $\SL(2,\mathbb{C})$-character is determined by the values 
\begin{align*}
(\tr A_1,\tr A_2,\tr A_3,\tr A_1A_2,\tr A_2A_3,\tr A_3A_1,\tr A_1A_2A_3)\in\mathbb{C}^7.
\end{align*}
Thus, the $\SL(2,\mathbb{C})$-character variety for $F_3$ may be embedded as a subvariety in $\mathbb{C}^7$. In particular, the character variety is a hypersurface defined by Fricke's relation \cite{MagRin} for matrices in $\GL(2,\mathbb{C})$  --- in fact for $\PGL(2,\mathbb{C})$ since the identity is homogeneous: given three matrices $A_1,A_2,A_3\in \GL(2,\mathbb{C})$, set $A_0=A_1A_2A_3$. Then,
\begin{align}  \label{fricke}
4\det A_0=&
(\tr A_0)^2+\tr A_1\cdot\tr A_2\cdot\tr A_3\cdot\tr A_0
+\tr A_1A_2\cdot\tr A_2A_3\cdot\tr A_3A_1\\
&+\frac{1}{2}
\sum_{\sym}%
\left\{\begin{array}{l}(\tr A_i)^2\cdot\det A_jA_k
-\det A_i\cdot\tr A_j\cdot\tr A_k\cdot\tr A_jA_k\\
+\det A_i\cdot(\tr A_jA_k)^2-\tr A_0\cdot\tr A_i\cdot\tr A_jA_k
\end{array}
\right\}.\nonumber
\end{align}
The symmetric sum here is taken over all possible choices for $\{i,j,k\}=\{1,2,3\}$; the factor of $\tfrac{1}{2}$ in \eqref{fricke} compensates for indices such as $(i,j,k)=(1,2,3)$ and $(1,3,2)$ giving repeated terms. The proof of \eqref{fricke} uses the fact that it extends to a relation on $M(2,\mathbb{C})$ which is quadratic in each entry of $A_i$.\newline
\newline
For our purposes, we study the \emph{relative character variety} consisting of \emph{type-preserv\-ing} representations of $F_3$ (i.e.\ where peripheral elements are parabolic). The peripheral elements of the thrice-punctured projective plane are conjugate to $A_1A_2$,  $A_2A_3$, $A_3A_1$ or their inverses, so we impose the constraints:
\begin{align}  \label{eq:cuspidal}
\tr A_1A_2=\tr A_2A_3=\tr A_3A_1=2.
\end{align}
Moreover, instead of considering $\SL(2,\mathbb{C})$ characters, we consider characters of $\SL^{\pm}(2,\mathbb{C})$-representations such that
\begin{align}  \label{eq:1sided}
\det A_1=\det A_2=\det A_3=-1.
\end{align}
This choice may seem a little unnatural in $\GL(2,\mathbb{C})$ since we can simply replace $A_k$ by $iA_k$ to recover a $\SL(2,\mathbb{C})$ character. We choose $\det A_k=-1$ because it is natural when restricting to Fuchsian representations. With this normalisation, Fuchsian representations $\rho:F_3\rightarrow \PSL^{\pm}(2,\mathbb{R})$ have real coefficients. Note that when lifting from $\PSL^{\pm}(2,\mathbb{R})$ to $\SL^{\pm}(2,\mathbb{R})$, one-sided curves necessarily have negative determinant.\newline 
\newline
Since $\det A_i=-1$, then $A_i$ and $A_i^{-1}$ are not conjugate (as they would be in $\SL(2,\mathbb{C})$),  we need to specify orientations on the simple closed curves representing their conjugacy classes. In figure~\ref{fig:3},
\begin{figure}[ht]
	\centerline{\includegraphics[scale=0.4]{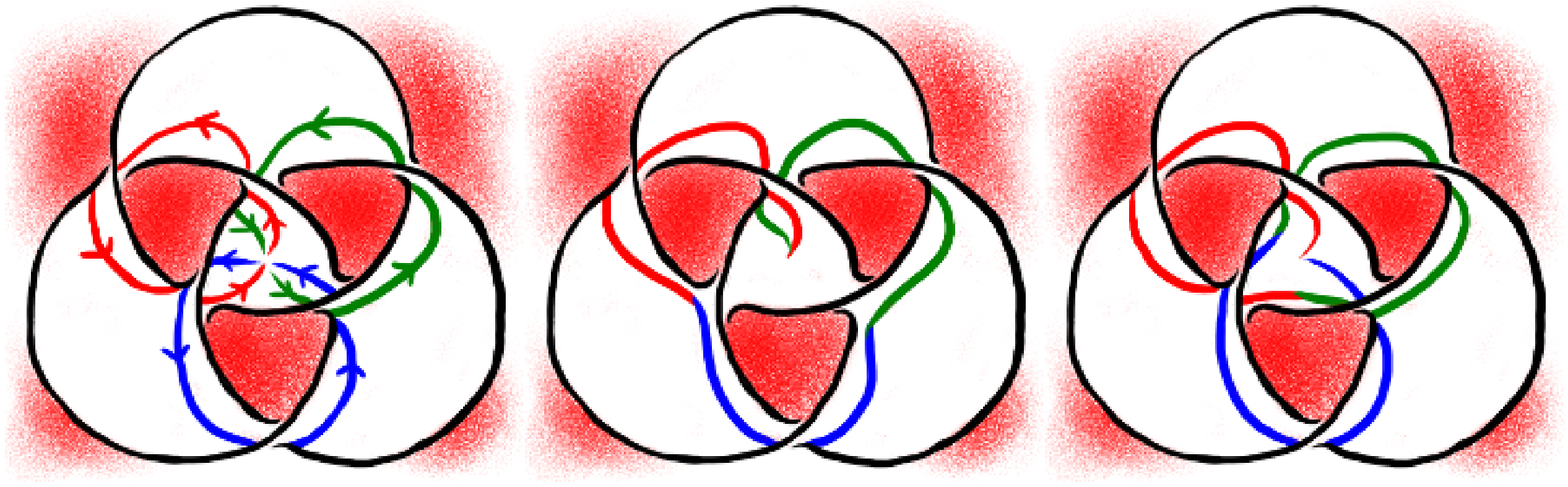}}
	\caption{Representative curves for ${\color{red}A_1},{\color{blue}A_2},{\color{green}A_3}$, ${\color{red}A_1}{\color{blue}A_2}{\color{green}A_3}$ and ${\color{red}A_1}{\color{green}A_3}{\color{blue}A_2}$.}
	\label{fig:3}
\end{figure}
we see that there is a choice of orientation for simple close curves representing $A_1$, $A_2$ and $A_3$ (anticlockwise) so that the curves representing $A_1A_2A_3$ and $A_1A_3A_2$ are simple.  These are the two choices of $A_1^{\pm 1}A_2^{\pm 1}A_3^{\pm 1}$ (up to conjugation and inversion) which are simple,  and we choose
\begin{align}
A_4=A_0^{-1}=(A_1A_2A_3)^{-1}\text{ and }A_4'=(A_1A_3A_2)^{-1}.
\end{align}
Set $a=\tr A_1$, $b=\tr A_2$, $c=\tr A_3$, $d=\tr A_4$ and $d'=\tr A_4'$, then \eqref{fricke} reorganises to yield  \eqref{quad}:
\begin{align*}
(a+b+c+d)^2=abcd\text{ and }(a+b+c+d')^2=abcd'
\end{align*}
which means that $(a,b,c,d)$ and $(a,b,c,d')$ are Markoff quads. In addition, since $d$ and $d'$ are the roots of the polynomial
\begin{align*}
p(x)=x^2+(2a+2b+2c-abc)x+(a+b+c)^2=(x-d)(x-d'),
\end{align*}
the following identities must hold:
\begin{align}\label{eq:edge}
d+d'+2a+2b+2c=abc\text{ and }dd'=(a+b+c)^2.
\end{align}
It should be noted that these are precisely the sum and product relations in \cite{GolTra}.

\begin{remark}
In \cite{MPTCha}, Maloni, Palesi and Tan study a different \emph{relative character variety} of representations of $F_3$ into $\SL(2,\mathbb{C})$ which arises from the four-punctured sphere --- their $A_i$ are constrained to be parabolic.  
\end{remark}

Successive applications of equation~\eqref{eq:edge} enables one to generate the trace (and hence the length) of every one-sided simple closed homotopy class on $X$. We now explain how to generate the traces of all of the 2-sided simple closed homotopy classes.\newline
\newline
Any two one-sided simple closed curves ${\color{red}\gamma_i},{\color{blue}\gamma_j}$ intersecting exactly once live inside an embedded punctured M\"obius strip, as depicted in Figure~\ref{fig:mob}.  They uniquely induce a two-sided simple closed curve as a boundary component (with the other boundary component peripheral in $N_{1,3}$).
\begin{figure}[ht]  
	\centerline{\includegraphics[height=2.5cm]{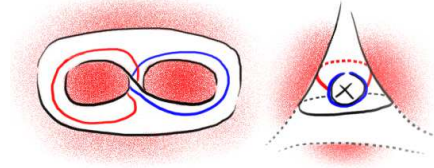}}
	\caption{Punctured M\"obius strip}
	\label{fig:mob}
\end{figure}

The following trace identity in $\GL(2,\mathbb{C})$
\begin{equation}  \label{eq:mobid}
\tr A_iA_j+\det A_j\cdot\tr A_iA_j^{-1}=\tr A_i\cdot\tr A_j
\end{equation} 
relates the complex lengths of peripheral curves $\alpha$ and $\beta$ of a punctured M\"obius strip to the complex lengths of ${\color{red}\gamma_i}$ and ${\color{blue}\gamma_j}$:
\begin{equation}   \label{eq:mobident} 
\cosh\left(\tfrac{1}{2}\ell_{\alpha}\right)+\cosh\left(\tfrac{1}{2}\ell_{\beta}\right)=2\sinh\left(\tfrac{1}{2}\ell_{\gamma_i}\right)\sinh\left(\tfrac{1}{2}\ell_{\gamma_j}\right).
\end{equation}
Since any 2-sided geodesic on a 3-cusped projective plane necessarily bounds a pair of pants, equation \eqref{eq:mobident} allows one to obtain the length of any 2-sided simple closed geodesic.

\subsection{The Curve Complex}
Equations~\eqref{eq:edge} and \eqref{eq:mobid} give us an algorithm to generate the entire length spectrum of a 3-cusped projective plane, starting from a corresponding Markoff quad. The combinatorics of this algorithm can be stored in terms of the \emph{curve complex} of $N_{1,3}$.\newline
\newline
Consider the geometric realisation of the abstract simplicial complex $\Omega^*$ with its $n$-simplices given by subsets of $n+1$ distinct homotopy classes of one-sided simple closed curves in $N_{1,3}$ that pairwise intersect once. Identifications of simplices as the faces of higher dimensional simplicies is given by inclusion. This is a pure simplicial $3$-complex, and its $1$-skeleton has been previously described by Scharlemann \cite{SchCom} as being the 1-skeleton of the cell complex formed from a tetrahedron by repeated stellar subdivision of the faces, but not the edges.\newline
\newline
The curve complex $\Omega$ that we're concerned with is the dual of $\Omega^*$. The decision to take the dual accords with Bowditch's conventions in \cite{BowPro,BowMar}. We now describe and assign notation for the cells of $\Omega$.\newline
\newline
\noindent
\textbf{The vertices, or $0$-cells of $\Omega$ are:}
\begin{align*}
\Omega^0:=\left\{
\{\alpha,\beta,\gamma,\delta\}\
\begin{array}{|l}
   \alpha,\beta,\gamma,\delta\text{ are homotopy classes of one-sided simple closed}\\ 
  \text{curves that pairwise geometrically intersect once}
\end{array}
\right\}
\end{align*} 
\textbf{The edges, or $1$-cells of $\Omega$ are:}
\begin{align*}
\Omega^1:=\left\{
\{\alpha,\beta,\gamma\}\ 
\begin{array}{|l}
 \alpha,\beta,\gamma\text{ are homotopy classes of one-sided simple closed}\\ 
 \text{curves that pairwise geometrically intersect once}
 \end{array}
 \right\}
\end{align*} 
Observe that each edge may be interpreted as a flip from one $0$-cell to another. Hence, the 1-skeleton of $\Omega$ is a 4-regular tree (i.e.\ each vertex has degree 4). Further,  the connectedness of this cell-complex described in \cite[Theorem~3.1]{SchCom} means that flips generate all possible $0$-cells, and hence all one-sided simple closed geodesics.\newline
\newline
\textbf{The faces, or $2$-cells of $\Omega$ are:}
\begin{align*}
\Omega^2:=\left\{
\{\alpha,\beta\}\
\begin{array}{|l}
\alpha,\beta\text{ are homotopy classes of one-sided simple closed curves}\\ 
\text{that intersect geometrically once}
\end{array}
\right\}
\end{align*} 
It follows from the observation in the previous subsection regarding punctured M\"obius strips embedded in $S$ that each face represents a unique homotopy class of essential, non-peripheral two-sided simple closed curves on $N_{1,3}$.\newline
\newline
\textbf{The 3-cells $\Omega^3$ of $\Omega$ are:}
\begin{align*}
\Omega^3:=\left\{\ \{\alpha\}\mid \alpha\text{ is an homotopy class of one-sided simple closed curves }\right\}
\end{align*} 
We later sometimes denote $3$-cells by capital letters, and use:
\begin{align*}
\vec{\Omega}^1=\left\{ \vec{e}=\{\alpha,\beta,\gamma;\delta'\rightarrow\delta\}\mid \{\alpha,\beta,\gamma\}\in\Omega^1\right\}.
\end{align*}
to denote the collection of \emph{oriented edges} of $\Omega$. In particular, $\{\alpha,\beta,\gamma;\delta'\rightarrow\delta\}$ points from $\{\delta'\}$ to $\{\delta\}$. Figure~\ref{fig:complex} illustrates the local geometry of an oriented edge.
\begin{figure}[ht]
\begin{center}
\includegraphics[scale=.75]{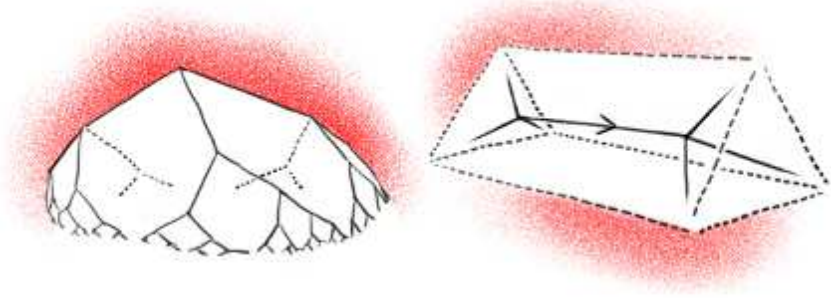}	
\end{center}
\caption{A $3$-cell (left) and an oriented edge (right).}
\label{fig:complex}
\end{figure}

\subsection{Markoff maps and characters}

Given a representation $\rho:F_3\to \GL(2,\mathbb{C})$ satisfying the trace condition \eqref{eq:cuspidal} and the determinant condition \eqref{eq:1sided}, we use Greek letters for simple closed curves and the corresponding Latin letters for the trace of the image of any homotopy class it defines. We decorate $\Omega$ with trace data by assigning to every 3-cell $\{\alpha\}\in\Omega^3$ its corresponding trace $\tr\rho(\alpha)=a$ thus defining a function:
\begin{align*}
\phi:\Omega^3\rightarrow\mathbb{C}\text{ by }\phi(\alpha)=\tr\rho(\alpha).
\end{align*}
A Markoff map induced from $\rho$ may be thought of as the character corresponding to $\rho$ restricted to one-sided simple closed homotopy classes. Our previous discussions in subsection~\ref{sec:trid} mean that the data carried by a Markoff map suffices to recover the whole character.\newline
\newline
We introduce the language of Markoff maps to allude to Bowditch's work \cite{BowPro,BowMar}; lower-dimensional simplices in $\Omega$ may be interpreted as mnemonics for encoding the following relations:\newline
\newline
\textbf{Vertex relation}: for $\{\alpha,\beta,\gamma,\delta\}\in\Omega^0$,  \eqref{quad} is equivalent to:
\begin{equation}  \label{vertrel}
\frac{d}{a+b+c+d}=\frac{a+b+c+d}{abc}.
\end{equation}
where $a=\tr\rho(\alpha)$, $b=\tr\rho(\beta)$, $c=\tr\rho(\gamma)$ and $d=\tr\rho(\delta)$. Note that the set of values of any four 3-cells which meet at a vertex corresponds to a Markoff quad.\newline
\newline
\textbf{Edge relation}: an edge $e=(\alpha,\beta,\gamma)\in\Omega^1$ lies in the intersection of the two 0-cells $(\alpha,\beta,\gamma,\delta)$ and $(\alpha,\beta,\gamma,\delta')$, and \eqref{eq:edge} yields:
\begin{align}  \label{edgerel}
\frac{a+b+c+d}{abc}+\frac{a+b+c+d'}{abc}=1,
\end{align}
where $d'=\tr\rho(\delta')$ and the others are as previously defined. Since each edge joins two vertices, the edge relation therefore tells us how to flip from one Markoff quad to another.\newline
\newline
\textbf{Face relation}: given $\{\alpha,\beta\}\in\Omega^2$ and $\epsilon$ the unique non-peripheral two-sided simple closed homotopy class disjoint from $\alpha$ and $\beta$, from \eqref{eq:mobid} we have:
\begin{align} \label{facerel}
ab=e+2
\end{align}
where $a=\tr\rho(\alpha)$, $b=\tr\rho(\beta)$ and $e=\tr\rho(\epsilon)$.\newline
\newline
We stress once again that these three relations allow us to generate the character for $\rho$ from a starting Markoff quad: the vertex and edge relations generate the traces for all the one-sided simple closed curves and the face relation then produces the traces for all of the two-sided simple closed curves.\newline
\newline
Thus, we're led to consider general maps $\phi:\Omega^3\rightarrow \mathbb{C}$ satisfying the edge and vertex relations. We call such functions \emph{Markoff maps}, and let $\Phi$ denote the collection of all Markoff maps.  In keeping with our notation for representations, we use Greek and Latin letters respectively for 3-cells (one-sided curves) and their image under some $\phi\in\Phi$.
\begin{lem}\label{lem:markoff}
The collection of Markoff maps and the collection of characters induced from $\SL^{\pm}(2,\mathbb{C})$-representations satisfying \eqref{eq:cuspidal} and \eqref{eq:1sided} are in canonical bijection.
\end{lem}

\begin{proof}
The restriction of any such character $\xi:F_3\rightarrow\mathbb{C}$ to the one-sided simple closed homotopy classes may be thought of as a Markoff map. Hence, it suffices to show that every Markoff map may be induced by a $\SL^{\pm}(2,\mathbb{C})$-representations satisfying \eqref{eq:cuspidal} and \eqref{eq:1sided}.\newline
\newline
Given a Markoff map $\phi:\Omega^3\rightarrow\mathbb{C}$, if there is a 3-cell $\{\alpha\}$ on which $\phi(\{\alpha\})=0$, fix an arbitrary 0-cell $\{\alpha,\beta,\gamma,\delta\}$ lying on the boundary of $\{\alpha\}$. Then consider the representation:
\begin{align}
\rho:F_3&=\left\langle \alpha,\beta,\gamma\right\rangle\rightarrow\SL^{\pm}(2,\mathbb{C})\\
\alpha&\mapsto
\left[\begin{matrix}
0 & 1\\
1 & 0
\end{matrix}\right],
\beta\mapsto
\left[\begin{matrix}
b & 1\\
1 & 0
\end{matrix}\right],
\gamma\mapsto
\left[\begin{matrix}
0 & 1\\
1 & c
\end{matrix}\right].\notag
\end{align}
It is easy to check that $\rho$ satisfies the desired trace and determinant conditions and that it induces $\phi$. If $\phi$ is nowhere-zero, choose an arbitrary 0-cell $\{\alpha,\beta,\gamma,\delta\}$, since $\phi$ is nowhere-zero,
\begin{align*}
(a+b+c+d)^2=abcd\neq0.
\end{align*}
This in turn means that the following representation is well-defined:
\begin{align}\label{eq:representation}
\rho:F_3&=\left\langle \alpha,\beta,\gamma\right\rangle\rightarrow\SL^{\pm}(2,\mathbb{C})\\
\alpha&\mapsto
\frac{1}{a+b+c+d}\left[\begin{matrix}
ab & b(a+c)\\
a(a+d) & a(a+c+d)
\end{matrix}\right],\notag\\
\beta&\mapsto
\frac{1}{a+b+c+d}\left[\begin{matrix}
ab & -b(b+d)\\
-a(b+c) & b(b+c+d)
\end{matrix}\right],\notag\\
\gamma&\mapsto
\frac{1}{a+b+c+d}\left[\begin{matrix}
ab+c(a+b+c+d) & b(a+c)\\
-a(b+c) & -ab
\end{matrix}\right].\notag
\end{align}
We omit the check to see that $\rho$ satisfies \eqref{eq:cuspidal} and \eqref{eq:1sided} and induces $\phi$.
\end{proof}

The above lemma means that the relative character variety that we wish to study may be characterised as the variety of Markoff maps. Among this collection of Markoff maps, we wish to focus on those which correspond to characters of quasi-Fuchsian representations. This leads us to define BQ-Markoff maps:\newline
\newline
For $k\geq0$ and $\phi\in\Phi$, define the set $\Omega_\phi^3(k)\subseteq \Omega^3$ by:
\begin{align}
\Omega_{\phi}^3(k):=\left\{\{\alpha\}\in\Omega^3\mid |\phi(\{\alpha\})|=|a|\leq k \right\}.
\end{align}
This set allows us to keep track of one-sided simple curves with trace less than $k$, and we similarly define $\Omega^2_\phi(k)\subset\Omega^2$ for two-sided simple curves. Every two-sided simple curve corresponds to a unique 2-cell $\{\alpha,\beta\}$ --- the shared face of the two 3-cells $\{\alpha\}$ and $\{\beta\}$. We define:
\begin{align}
\Omega^2_\phi(k):=\left\{\{\alpha,\beta\}\in\Omega^2\mid |\phi(\{\alpha\})\phi(\{\beta\})|=|ab|\leq k\right\}.
\end{align}
Note that in using $|ab|$ instead of $|ab-2|$ for the conditions imposed, the set $\Omega^2_\phi(k)$ doesn't quite correspond to the set of two-sided simple curves with trace less than $k$. Although we will find this definition more suited to our analysis. In addition, we later focus on the following collection of Markoff maps $\Phi_{BQ}\subset\Phi$:
\begin{align}  \label{BQ}
\Phi_{BQ}:=\left\{
\phi\in\Phi\ 
\begin{array}{|l}
 \Omega^2_\phi(k)\text{ is finite for any }k,\\
 \text{and for any }\{\alpha,\beta\}\in\Omega^2_\phi(4),\,ab\notin[0,4]
\end{array}
\right\}.
\end{align}
We show in section~\ref{sec:mcshane} that these are sufficient conditions to guarantee the existence of a McShane identity for a given Markoff map. These conditions are similar to Bowditch's BQ-condition, which is a conjectural trace-based characterisation of quasi-Fuchsian representations. The following result shows that our condition is also necessary for quasi-Fuchsian representations of the thrice-punctured projective plane.

\begin{lem}
Markoff maps obtained from quasi-Fuchsian representations lie in $\Phi_{BQ}$.
\end{lem}

\begin{proof}
Given a Markoff map $\phi$ arising from a quasi-Fuchsian representation $\rho$, consider the multiset of complex numbers obtained from evaluating $\phi$ on $\Omega^3_\phi(m)$. Since this multiset is a subset of the simple trace spectrum of $\rho$, which is obtained (up to sign) from taking $2\sinh(\tfrac{1}{2}\cdot)$ of the simple length spectrum, the discreteness of the simple length spectrum ensures that $\Omega^3_\phi(m)$ is finite. Any 2-cell in $\Omega^2_\phi(m)$ is the intersection of precisely one pair of 3-cells, hence the cardinality of $\Omega^2_\phi(m)$ is bounded by the square of the cardinality of $\Omega^3_\phi(m)$ and is finite.\newline
\newline
Next, if $ab\in[0,4]$ for some $\{a,b\}\in\Omega^2$, then there is a 2-sided non-peripheral homotopy class $\epsilon$ whose trace is $ab-2\in[-2,2]$, thus contradicting the fact that quasi-Fuchsian representations have neither parabolics nor elliptics.
\end{proof}

\subsection{Markoff-Hurwitz numbers}\label{markoff-hurwitz} 
The \emph{Markoff-Hurwitz equation} is given by
\begin{align}  \label{MarkHur}
a_1^2+...+a_n^2=a_1...a_n.
\end{align}
Its quadratic nature means its solution variety admits a discrete group action generated by $a_i\mapsto a_1...a_n/a_i-a_i$, i.e.\ one can obtain new solutions from old via flips.  Previously, only the $n=3$ solutions of the Markoff-Hurwitz equation were known to have a length spectrum interpretation---for a hyperbolic punctured torus.  One observation of this paper is the association of solutions of the $n=4$ solutions of the Markoff-Hurwitz equation with Markoff quads, and hence the $\SL^{\pm}(2,\mathbb{C})$ relative character variety of the thrice-punctured projective plane. The substitution $a=a_1^2$, $b=a_2^2$, $c=a_3^2$, $d=a_4^2$ into \eqref{quad} defines a map between solutions $(a_1,\ldots,a_4)$ of \eqref{MarkHur} for $n=4$ and Markoff quads.  The map is a quotient by the $Z_2^3$-action $(a_1,a_2,a_3,a_4)\mapsto(\pm a_1,\pm a_2,\pm a_3,\pm a_4)$ (even number of minus signs) on solutions of the Markoff-Hurwitz equation.  Flips of Markoff quads correspond to flips of solutions of the Markoff-Hurwitz equation.

\subsection{Teichm\"uller space}\label{sec:teich}
The remainder of this section deals with Markoff maps corresponding to Fuchsian representations. From the proof of lemma~\ref{lem:markoff}, we see that these are precisely the real Markoff maps. We focus on the Teichm\"uller component of the real relative character variety, showing that it consists of the positive real Markoff maps.

The Teichm\"uller space $\mathcal{T}(S)$ of a surface $S$ encodes all the ways of assigning a complete finite-area hyperbolic metric to $S$, up to homotopy. Concretely, it may be expressed as:
\begin{align*}
\mathcal{T}(S):=\{\, (X,f)\mid f:S\rightarrow X\text{ is a homeomorphism }\}/\sim
\end{align*}
where $(X_1,f_1)\sim(X_2,f_2)$ if and only if
\begin{align*}
f_2\circ f_1^{-1}:X_1\rightarrow X_2
\end{align*}
is homotopy equivalent to a hyperbolic isometry. We denote these equivalence classes, or \emph{marked surfaces}, by $[X,f]$.\newline
\newline
An \emph{ideal triangulation} of $S$ is, up to homotopy, a triangulation of $S$ with vertices at the punctures of $S$. Given a marked surface $[X,f]$, the image $f(\sigma)$ of an arc $\sigma$ on $S$ pulls tight to a unique homotopy equivalent geodesic arc on $X$. Thus, any ideal triangulation on $S$ is represented by an (geodesic) ideal triangulation on $X$ --- a maximal collection of simple bi-infinite geodesic arcs with both ends up cusps. For our purposes, we restrict to ideal triangulations $\triangle$ on thrice-punctured projective planes $S$ representable by paths with \emph{distinct end points}.\newline
\newline
Horocycles of length $1$ around a cusp are always simple on a complete hyperbolic surface. Thus, given an ordered ideal triangulation $(\sigma_1,\sigma_2,\sigma_3,\tau_1,\tau_2,\tau_3)$ on $X$, we obtain lengths $(s_1,s_2,s_3,t_1,t_2,t_3)$ of these infinite geodesic arcs truncated at the three length $1$ horocycles bounding cusps $1,2,3$. The $\lambda$-lengths for $X$ with respect to this ordered ideal triangulation is then given by:
\begin{align*}
(\lambda_1,\lambda_2,\lambda_3,\mu_1,\mu_2,\mu_3)=(\exp\tfrac{1}{2}s_1,\exp\tfrac{1}{2}s_2,\exp\tfrac{1}{2}s_3,\exp\tfrac{1}{2}t_1,\exp\tfrac{1}{2}t_2,\exp\tfrac{1}{2}t_3).
\end{align*}
In \cite{PenDec}, Penner shows that these $\lambda$-lengths form global coordinates on the Teichm\"uller space of any punctured surface. This is also true for the Teichm\"uller space of punctured non-orientable surfaces.\newline
\newline
The following lemma is a topological correspondence which is promoted to a geometric correspondence below. 
\begin{lem}\label{th:bijection}
There is a natural bijection between $\Omega^0$ and
\begin{align*}
\left\{
\text{ the collection of ideal triangulations of }S\text{ with distinct end points }
\right\}
\end{align*}
given by sending $\{\alpha,\beta,\gamma,\delta\}\in\Omega^0$ to the unique (up to homotopy) ideal triangulation where each arc intersects precisely two of the geodesics in $\{\alpha,\beta,\gamma,\delta\}$.
\end{lem}

\begin{proof}
Any essential two-sided simple closed curve is a boundary component of a thickening of a unique pair of once-intersecting one-sided simple closed curves (figure~\ref{fig:mob}), and a boundary component of a thickening of a unique arc joining distinct punctures.  By alternately thinking of a two-sided curve as boundary components of these two thickenings, we see that pairs of intersection points between two 2-sided simple closed curves correspond to single intersection points between two one-sided simple closed curves and between two arcs joining distinct punctures (where the punctures count as single intersection points). Hence the six arcs obtained in this way  are disjoint outside the punctures if and only if the homotopy classes in
$\{\alpha,\beta,\gamma,\delta\}\in\Omega^0$ pairwise intersect exactly once.
\end{proof}
\begin{figure}[ht]
\begin{center}
\includegraphics[scale=.6]{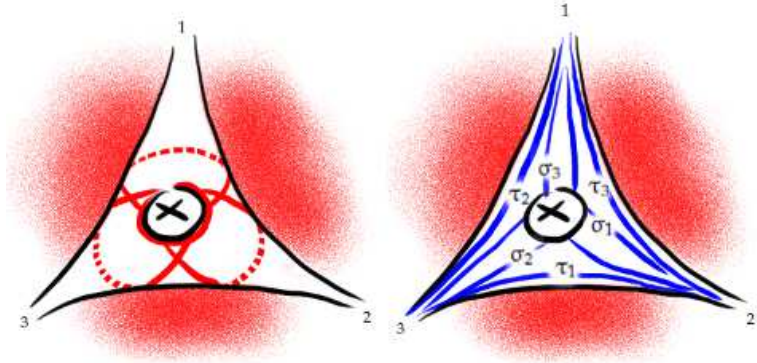}	
\end{center}
\caption{A 4-tuple of {\color{red}curves} corresponding to a {\color{blue}triangulation}.}
\label{fig:triangulation}
\end{figure}
\begin{lem}
The $\lambda$-lengths for an ideal triangulation $\triangle$ of $N_{1,3}$ identifies the Teichm\"ull\-er space $\mathcal{T}(N_{1,3})$ as:
\begin{align*}
\left\{
\hspace{-1mm}\begin{array}{r|l}&\mu_1\mu_2\mu_3+\mu_1\lambda_2\lambda_3+\lambda_1\mu_2\lambda_3+\lambda_1\lambda_2\mu_3=\lambda_1\lambda_2\mu_1\mu_2\\
(\lambda_1,\lambda_2,\lambda_3,\mu_1,\mu_2,\mu_3)\in\mathbb{R}_+^6 
 & \mu_1\mu_2\mu_3+\mu_1\lambda_2\lambda_3+\lambda_1\mu_2\lambda_3+\lambda_1\lambda_2\mu_3=\lambda_1\lambda_3\mu_1\mu_3\\
  & \mu_1\mu_2\mu_3+\mu_1\lambda_2\lambda_3+\lambda_1\mu_2\lambda_3+\lambda_1\lambda_2\mu_3=\lambda_2\lambda_3\mu_2\mu_3
\end{array}
\right\}.
\end{align*}
\end{lem}
These $\lambda$-lengths of an ideal triangulation may be expressed in terms of the Markoff quad of the associated quadruple of one-sided geodesics in $\Omega^0$ corresponding to the  used to the define these $\lambda$-lengths. 
\begin{align}  \label{quambda}
&(a,b,c,d)=
\left(\frac{\lambda_2\lambda_3}{\lambda_1}, \frac{\lambda_1\lambda_3}{\lambda_2},\frac{\lambda_1\lambda_2}{\lambda_3}, \frac{\mu_1\mu_2}{\lambda_3}=\frac{\mu_1\mu_3}{\lambda_2}=\frac{\mu_2\mu_3}{\lambda_1}\right),\\
&(\lambda_1,\lambda_2,\lambda_3,\mu_1,\mu_2,\mu_3)=(\sqrt{bc},\sqrt{ac},\sqrt{ab},\sqrt{ad},\sqrt{bd},\sqrt{cd}\,).
\notag
\end{align}

Thus, we may also use positive Markoff quads to globally parametrise the Teichm\"uller space:
\begin{prop}   \label{th:teich}
Given an ordered 4-tuple $(\alpha,\beta,\gamma,\delta)$ intersecting a fixed triangulation on $N_{1,3}$ as per figure~\ref{fig:triangulation}, then the map
\begin{align*}
\mathcal{T}(N_{1,3})&\rightarrow \left\{(a,b,c,d)\in\mathbb{R}_+^4\mid(a+b+c+d)^2=abcd\right\}\\
[X,f]&\mapsto (2\sinh\tfrac{1}{2}\ell_{\alpha}(X),2\sinh\tfrac{1}{2}\ell_{\beta}(X),2\sinh\tfrac{1}{2}\ell_{\gamma}(X),2\sinh\tfrac{1}{2}\ell_{\delta}(X))
\end{align*}
is a real-analytic diffeomorphism, where $\ell_{\alpha}(X)$ denotes the length of the geodesic representative of $f_*(\alpha)$ on $X$. We call these global coordinates the \emph{trace coordinates} for $\mathcal{T}(N_{1,3})$.
\end{prop}
\begin{proof}
With a little hyperbolic trigonometry and successive applications of the ideal Ptolemy relation \cite{HuaMod}, we can show that \eqref{quambda} explicitly gives the desired diffeomorphism between the trace coordinates and the $\lambda$-coordinates for $\mathcal{T}(N_{1,3})$. The fact that this map is real-analytic is then a simple consequence of the real-analyticity of the $\lambda$-lengths.
\end{proof}
\begin{cor}
The set of positive Markoff quads is the Teichm\"uller component of the real character variety.
\end{cor}
\begin{proof}
Proposition~\ref{th:teich} proves that the set of positive Markoff quads is real-analytic\-ally diffeomorphic to Teichm\"uller space.  It remains to show that it is a connected component of the real character variety.  Suppose that one of the coordinates vanishes, say $d=0$.  Then by \eqref{quad}, $a+b+c=0$.  But if this point lies in the limit of a path in the set of \emph{positive} Markoff quads then each of $a$, $b$ and $c$ must tend to 0 along the path.  In particular, at some point on the path $abcd<256$. But this contradicts \eqref{quad} since
\begin{align*}
(a+b+c+d)^2\geq 16\sqrt{abcd}>abcd
\end{align*}
where the first inequality is the arithmetic mean-geometric mean inequality.
\end{proof}

\subsection{The mapping class group.}   \label{sec:mcg}
Markoff quads are points on the $\SL^{\pm}(2,\mathbb{C})$ relative character variety of the thrice-punctured projective plane, which is the hypersurface in $\mathbb{C}^4$ defined by equation~\eqref{quad}. The transformation \eqref{flip},
\begin{align*} 
(a,b,c,d)\mapsto(a,b,c,d'=abc-2a-2b-2c-d),
\end{align*}
combined with the following even permutations
\begin{align}
(a,b,c,d)\mapsto(b,a,d,c),(c,d,a,b),(d,c,b,a)
\end{align}
generate the pure mapping class group of the thrice-punctured projective plane, and specify its action on the corresponding relative character variety. \emph{Pure} here means that we restrict to elements of the mapping class group that fix punctures.\newline
\newline
We construct an explicit homeomorphism $f_4:N_{1,3}\rightarrow N_{1,3}$ that takes $(\alpha,\beta,\gamma,\delta)$ to $(\alpha,\beta,\gamma,\delta')$. Consider the hexagonal fundamental domain of $N_{1,3}$ obtained by cutting along $\sigma_1,\sigma_2$ and $\sigma_3$.
\begin{figure}[ht]
\begin{center}
\includegraphics[scale=.6]{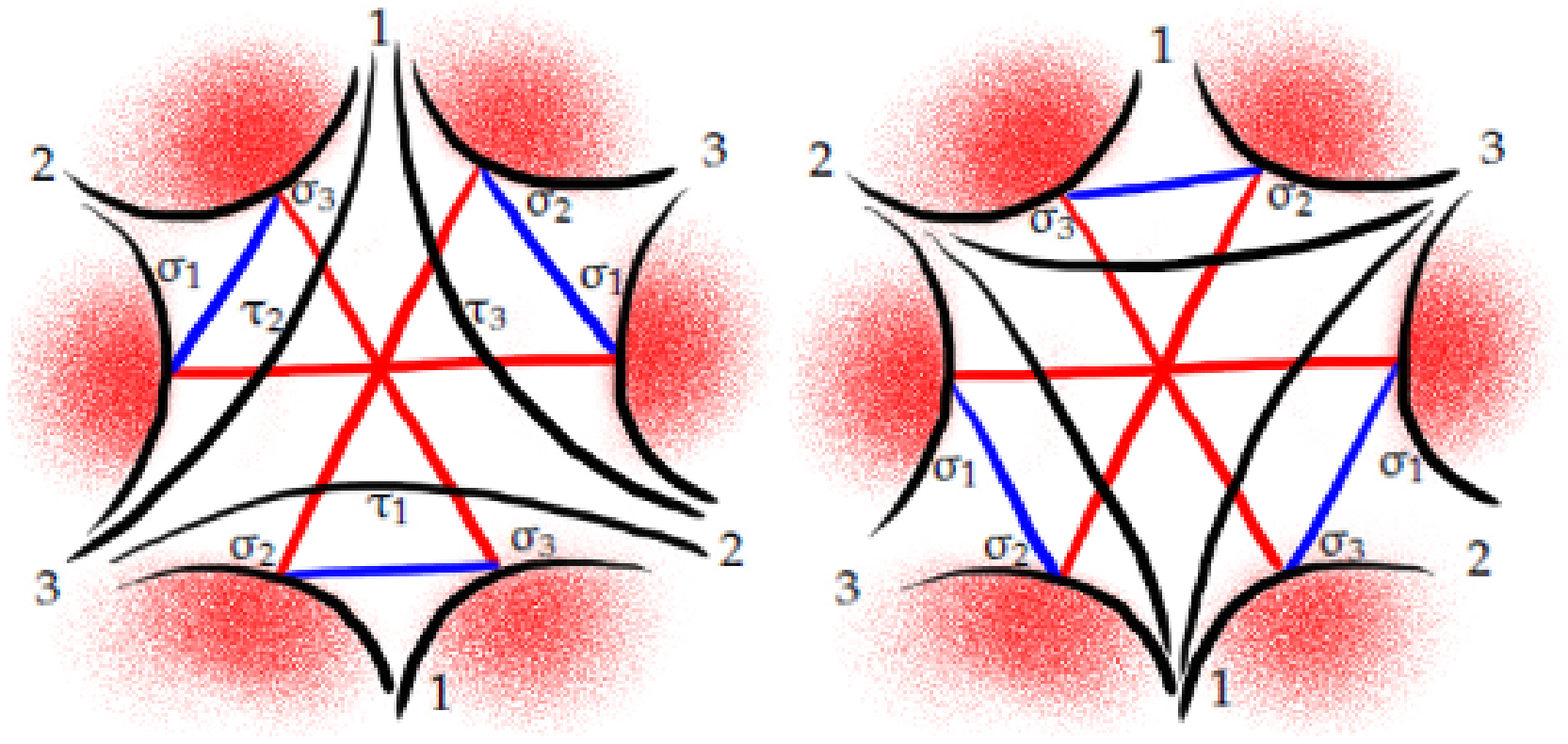}	
\end{center}
\caption{The map $f_4$ fixing ${\color{red}\alpha},{\color{red}\beta},{\color{red}\gamma}$, but switching ${\color{blue}\delta}$ and ${\color{blue}\delta'}$.}
\label{fig:flipmap}
\end{figure}
From figure~\ref{fig:flipmap}, we see that a rotation by $\pi$ of this fundamental domain fixes the labeling of the punctures and fixes each of $\alpha$, $\beta$ and $\gamma$ whilst taking $\alpha$ to $\alpha'$. The action of the mapping class $[f_4]\in\Gamma(N_{1,3})$ therefore takes the Markoff quad $(a,b,c,d)$ corresponding to a marked surface $[X,f]$ to
\begin{align*}
[f_4](a,b,c,d)=(a,b,c,abc-2a-2b-2c-d),
\end{align*}
that is: $[f_4]$ corresponds to a flip in the fourth entry. By symmetry, there are four flips $[f_1],[f_2],[f_3],[f_4]\in\Gamma(N_{1,3})$ which flip the corresponding entries of $(a,b,c,d)$. Let $F\leq\Gamma(N_{1,3})$ denote the subgroup generated by these four flips.

\begin{lem}
$F\cong\mathbb{Z}_2\ast\mathbb{Z}_2\ast\mathbb{Z}_2\ast\mathbb{Z}_2$, where each $\mathbb{Z}_2$ is generated by one of the $[f_i]$.
\end{lem}

\begin{proof}
First observe that each $[f_i]$ is indeed order $2$. To see that there are no other relations, consider the action of a reduced string of flips on the 1-skeleton of the curve complex: since the 1-skeleton is a 4-regular tree, performing each flip in a sequence of flips necessarily takes us farther from the origin.
\end{proof}

We now consider a different subgroup in $\Gamma(N_{1,3})$: the stabiliser of $\{\alpha,\beta,\gamma,\delta\}$. Due to lemma~\ref{th:bijection}, this subgroup must also stabilise $\triangle$ --- the triangulation corresponding to $\{\alpha,\beta,\gamma,\delta\}$.

\begin{lem}
$\mathrm{Stab}(\triangle)\cong\mathbb{Z}_2\times\mathbb{Z}_2$.
\end{lem}

\begin{proof}
There are four triangles $T_1,T_2,T_3,T_4$ induced by the triangulation $\triangle$ on $N_{1,3}$, and any element of $\mathrm{Stab}(\triangle)$ must take $T_1$ to one these four triangles. Since there is a unique way to map $T_1$ to any of these four triangles so as to preserve puncture-labeling, knowing the image of $T_1$ determines the entire mapping class. By symmetry, these mapping classes must have the same order, hence $\mathrm{Stab}(\triangle)$ is the Klein four-group.
\end{proof}

This stabiliser subgroup is given by:
\begin{align*}
\left\{[\mathrm{id}],
\begin{array}{l}
[\varphi_1]:(a,b,c,d)\mapsto (b,a,d,c),\\
{[\varphi_2]}:(a,b,c,d)\mapsto(c,d,a,b), \\
{[\varphi_3]}:(a,b,c,d)\mapsto (d,c,b,a),
\end{array}
\right\}
\end{align*}
when thought of as acting on the trace coordinates for $\mathcal{T}(N_{1,3})$.

\begin{lem}
The subgroups $F$ and $\mathrm{Stab}(\triangle)$ generate the whole mapping class group $\Gamma(N_{1,3})$.
\end{lem}

\begin{proof}
Given an arbitrary element $[h]\in\Gamma(N_{1,3})$, the action of $[h]$ on $(a,b,c,d)=[X,f]\in\mathcal{T}(N_{1,3})$ produces another Markoff quad $(\bar{a},\bar{b},\bar{c},\bar{d})$ corresponding to the traces of $(h_*(\alpha), h_*(\beta),h_*(\gamma),h_*(\delta)$). Since the four flips $[f_i]$ generate all Markoff quads associated to the Fuchsian representation for $X$, there is an element $[g]\in F$ such that $[g]\circ[h]=[g\circ h]$ simply permutes $a,b,c,d$. By choosing $X$ to be a surface where there are only four simple one-sided geodesics with traces $\{a,b,c,d\}$ (e.g.: the $(4,4,4,4)$ surface), we see that $[g\circ h]\in\mathrm{Stab}(\triangle)$.
\end{proof}

\begin{lem}
$F$ is a normal subgroup of $\Gamma(N_{1,3})$.
\end{lem}

\begin{proof}
Note that it suffices to show that $\mathrm{Stab}(\triangle)$ preserves $\{[f_1],[f_2],[f_3],[f_4]\}$. We perform this check for $[f_1]$, the rest follow by symmetry:
\[
[\varphi_1]^{-1}\circ [f_1]\circ [\varphi_1]=[f_2],\,[\varphi_2]^{-1}\circ [f_1]\circ [\varphi_2]=[f_3]\text{, and }[\varphi_3]^{-1}\circ [f_1]\circ [\varphi_3]=[f_4].
\]
\end{proof}
Since $F$ and $\mathrm{Stab}$ generate $\Gamma(N_{1,3})$ and their intersection is the trivial group, we obtain the following result:
\begin{thm}
$\Gamma(N_{1,3})=F\rtimes\mathrm{Stab}(\triangle)\cong(\mathbb{Z}_2\ast\mathbb{Z}_2\ast\mathbb{Z}_2\ast\mathbb{Z}_2)\rtimes(\mathbb{Z}_2\times\mathbb{Z}_2)\cong\mathbb{Z}_2\ast(\mathbb{Z}_2\times\mathbb{Z}_2)$. In particular:
\begin{align*}
\Gamma(N_{1,3})&\cong
\left\langle
\begin{array}{r|l} 
f_1,f_2,f_3,f_4, & f_1^2=f_2^2=f_3^2=f_4^2=g^2=h^2=1,gh=hg\\
g,h & g^{-1}f_1g=f_2,h^{-1}f_1h=f_3,g^{-1}f_3g=f_4
\end{array}
\right\rangle\\
&\cong\left\langle f,g,h\mid f^2=g^2=h^2=1,gh=hg\right\rangle.
\end{align*}
\end{thm}

\subsection{The moduli space}
Recall that the moduli space $\mathcal{M}(N_{1,3})$ of hyperbolic structures on $N_{1,3}$ is given by $\mathcal{T}(N_{1,3})/\Gamma(N_{1,3})$. Since $F$ is a normal subgroup of $\Gamma(N_{1,3})$, the space $\mathcal{T}(N_{1,3})/F$ must be a finite cover of $\mathcal{M}(N_{1,3})$. To better see what $\mathcal{T}(N_{1,3})/F$ looks like, we first define another global coordinate chart for $\mathcal{T}(N_{1,3})$. 
\begin{lem}\label{lem:horo}
The Teichm\"uller space $\mathcal{T}(N_{1,3})$ may be real-analytically identified with the following (open) 3-simplex:
\begin{align*}
\{(H_a,H_b,H_c,H_d)\in\mathbb{R}_+^4\mid H_a+H_b+H_c+H_d=1\},
\end{align*}
we call this the \emph{horocyclic coordinate} for $\mathcal{T}(N_{1,3})$.
\end{lem}
\begin{proof}
The explicit diffeomorphisms between the horocyclic coordinates and the trace coordinates is given as follows:
\begin{align*}
H_a=\sqrt{\frac{a}{bcd}}=\frac{a}{a+b+c+d},\,&H_b=\sqrt{\frac{b}{acd}}=\frac{b}{a+b+c+d}\\
H_c=\sqrt{\frac{c}{abd}}=\frac{c}{a+b+c+d},\,&H_d=\sqrt{\frac{d}{abc}}=\frac{d}{a+b+c+d},
\end{align*}
and the inverse map is given by:
\[
a=\sqrt{\frac{H_a}{H_bH_cH_d}},\,b=\sqrt{\frac{H_b}{H_aH_cH_d}},\,c=\sqrt{\frac{H_c}{H_aH_bH_d}},\,d=\sqrt{\frac{H_d}{H_aH_bH_c}}.
\]
\end{proof}
\begin{remark}
The horocyclic coordinates are so named because they correspond to the lengths of horocyclic segments on the length $1$ horocycles at the cusps of a marked surface $[X,f]$. Coupled with the labeling in figure~\ref{fig:triangulation}, figure~\ref{fig:horocycle} illustrates this correspondence.
\end{remark}
\begin{figure}[ht]
\begin{center}
\includegraphics[scale=.5]{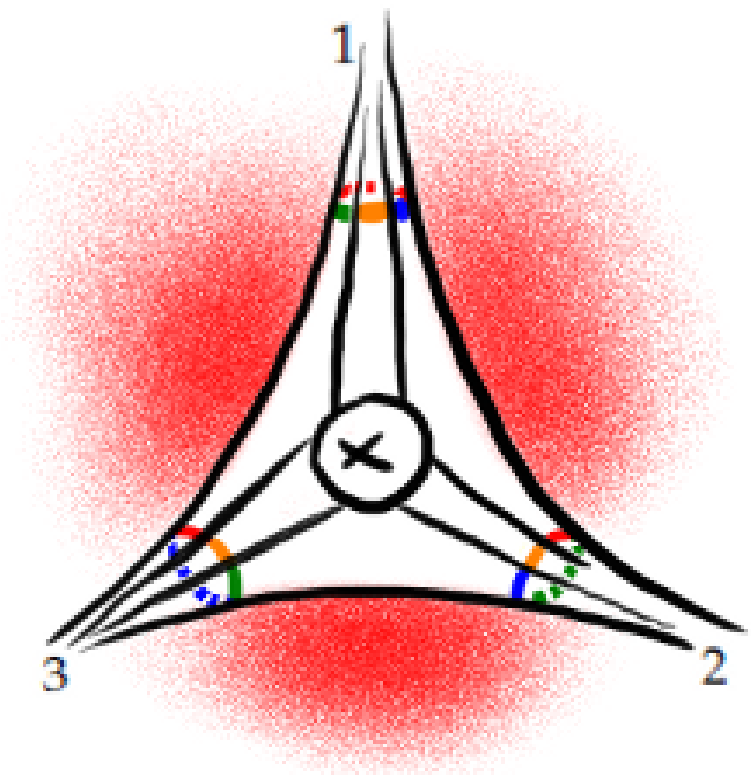}	
\end{center}
\caption{Horocyclic segments (each) of length ${\color{red}H_a},{\color{blue}H_b},{\color{green}H_c},{\color{orange}H_d}$.}
\label{fig:horocycle}
\end{figure}

\begin{thm}\label{thm:moduli}
The moduli space $\mathcal{M}(N_{1,3})$ of a thrice-punctured projective plane is homeomorphic to an open $3$-ball with an open hemisphere of order 2 orbifold points glued on, and a line of orbifold points running straight through the center of this $3$-ball ---  joining two antipodal points of this orbifold hemisphere. The orbifold points on this line are of order $2$, except for the very center point of the $3$-ball, which is order 4.
\end{thm}

\begin{proof}
In the horocyclic coordinates described in lemma~\ref{lem:horo}, the flips generating $F$ act as follows:
\begin{align*}
[f_1]&:(H_a,H_b,H_c,H_d)\mapsto(1-H_a,H_b\frac{H_a}{1-H_a},H_c\frac{H_a}{1-H_a},H_d\frac{H_a}{1-H_a}),\\
[f_2]&:(H_a,H_b,H_c,H_d)\mapsto(H_a\frac{H_b}{1-H_b},1-H_b,H_c\frac{H_b}{1-H_b},H_d\frac{H_b}{1-H_b}),\\
[f_3]&:(H_a,H_b,H_c,H_d)\mapsto(H_a\frac{H_c}{1-H_c},H_b\frac{H_c}{1-H_c},1-H_c,H_d\frac{H_c}{1-H_c}),\\
[f_4]&:(H_a,H_b,H_c,H_d)\mapsto(H_a\frac{H_d}{1-H_d},H_b\frac{H_d}{1-H_d},H_c\frac{H_d}{1-H_d},1-H_d).
\end{align*}

From this, we see that the fixed points of $[f_1],[f_2],[f_3],[f_4]$ are respectively given by imposing the following conditions on the horocyclic coordinates:
\begin{align*}
H_a=\tfrac{1}{2},H_b=\tfrac{1}{2},H_c=\tfrac{1}{2},H_d=\tfrac{1}{2}.
\end{align*}
The region in $\mathcal{T}(N_{1,3})$ enclosed by these four planes is therefore a fundamental domain for $\mathcal{T}(N_{1,3})/F$. In this case, this fundamental domain is an octahedron. Since $[f_1]$ acts by swapping the two regions separated by $H_a=\tfrac{1}{2}$, the image of these fixed points in $\mathcal{T}(N_{1,3})/F$ are order 2 (reflection) orbifold points. Similar comments hold for each of the $[f_i]$. Thus, $\mathcal{T}(N_{1,3})/F$ is an open octahedron with four triangles of order 2 orbifold points glued onto a collection of four non-adjacent sides.\newline
\newline
Finally, by noting that $\mathrm{Stab}(\triangle)$ acts on the horocyclic coordinates by:
\begin{align*}
[\varphi_1]&:(H_a,H_b,H_c,H_d)\mapsto(H_b,H_a,H_d,H_c),\\
[\varphi_2]&:(H_a,H_b,H_c,H_d)\mapsto(H_c,H_d,H_a,H_b),\\
[\varphi_3]&:(H_a,H_b,H_c,H_d)\mapsto(H_d,H_c,H_b,H_a).
\end{align*}
We obtain the desired result.
\end{proof}

\begin{remark} 
The interior $3$-ball of $\mathcal{M}(N_{1,3})$ may be geometrically interpreted as the set of 3-cusped projective planes which have a unique unordered 4-tuple of geodesics whose flips are strictly longer.
\end{remark}

\subsection{Integral Markoff quads}

To conclude this section, we characterise the positive integral Markoff quads. Positive integral Markoff triples are of importance in number theory. They arise in approximating real numbers \cite{CasInt}, the Markoff spectrum is closely related to the Lagrange spectrum and Markoff's theorem provides an integral Markoff triples-based characterisation of indefinite binary quadratic forms \cite{CohApp,MarSur}.

\begin{thm}
Every positive integer Markoff quad may be generated by a sequence of flips and coordinate permutations from precisely one of the following eight integer Markoff quads:
\begin{align}
\begin{array}{cccc}
(1,5,24,30),&(1,6,14,21),&(1,8,9,18),&(1,9,10,10),\\
(2,3,10,15),&(2,5,5,8),&(3,3,6,6),&(4,4,4,4).
\end{array}
\end{align}
\end{thm}

\begin{proof}
The edge relation(s)~\eqref{eq:edge} tell us that any flip on a positive integral Markoff quad $(a,b,c,d)$ results in another positive integral Markoff quad. Thus, by applying a sequence of flips, we may assume that $(a,b,c,d)$ lies in the fundamental domain of the moduli space described in the proof of theorem~\ref{thm:moduli}. Furthermore, up to reordering, we may assume wlog that $0<a\leq b\leq c\leq d$. The fact that $(a,b,c,d)$ lies in this fundamental domain means that $\frac{d}{a+b+c+d}\leq\frac{1}{2}$, and hence $d\leq a+b+c$. Thus,
\begin{gather*}
2(a+b+c)\geq a+b+c+d=\sqrt{abcd}\geq c\sqrt{ab},\\
\Rightarrow4c\geq2(a+b)\geq c(\sqrt{ab}-2),\\
\Rightarrow4\geq\sqrt{ab}-2\text{ and hence }36\geq ab.
\end{gather*}
Moreover, if $ab\leq4$, then substituting this into equation~\eqref{quad}:
\begin{gather*}
(a+b)^2+2(a+b)(c+d)+(c+d)^2=abcd\leq4cd\\
\Rightarrow (a+b)^2+(c-d)^2\leq0.
\end{gather*}
As this is impossible, we conclude that $5\leq ab\leq36$. If $a=1$, then $5\leq b\leq 36$ and hence $a+b\leq37$ and $c\geq d-37$. Thus:
\begin{gather*}
(37+2d)^2\geq(a+b+c+d)^2=abcd\geq5(d-37)d.
\end{gather*}
Solving for this quadratic over the integers shows that $1\leq d\leq 337$. This in turn also limits the possible values for $c\leq d$. Performing similar computations for $a=2,3,4$, we obtain the following cases:
\begin{align*}
\begin{array}{lll}
\text{if }a=1,& 5\leq b\leq 36, & 5\leq\max\{b,d-(a+b)\}\leq c\leq d\leq 337;\\
\text{if }a=2,& 3\leq b\leq 18, & 3\leq\max\{b,d-(a+b)\}\leq c\leq d\leq 101;\\
\text{if }a=3,& 3\leq b\leq 12, & 3\leq\max\{b,d-(a+b)\}\leq c\leq d\leq 40;\\
\text{if }a=4,& 4\leq b\leq 9, & 4\leq\max\{b,d-(a+b)\}\leq c\leq d\leq 26.
\end{array}
\end{align*}
It is unnecessary to consider the cases $a=5,6$ due to theorem~\ref{th:systole}, specifically: the fact that $(a,b,c,d)$ lie on the fundamental domain of the moduli space means that $a$ must be the trace of the systolic homotopy class on the surface corresponding to $(a,b,c,d)$, and we know that this trace can at most be equal to $4$.\newline
\newline
Checking through these possible values for $(a,b,c,d)$ on a computer then completes this proof.
\end{proof}

\section{Analysis on the Curve Complex}   \label{sec:cc}

Given a Markoff map $\phi$, for every edge $e=\{\alpha,\beta,\gamma\}$ fix an oriented edge 
\begin{align*}
\vec{e}=\{\alpha,\beta,\gamma;\delta'\rightarrow \delta\}\text{ to satisfy }|d'|\geq |d|,
\end{align*}
where as usual $(a,b,c,d)=(\phi(\alpha),\phi(\beta),\phi(\gamma),\phi(\delta))$. For most edges, this choice is canonical, and for edges with equality in $|d'|=|d|$, an arbitrary orientation is chosen. This produces an orientation on $\Omega^1$, where edges may be thought of as pointing from 3-cells corresponding to longer geodesics to 3-cells corresponding to shorter geodesics. Thus, analysis of the dynamics (in terms of the directions) of these edges informs us about the behaviour of geodesic length growth for $\phi$.\newline
\newline
The following lemma gives alternative algebraic characterisations of this trace comparison.

\begin{lem}
For a Markoff quad $(a,b,c,d)\in\mathbb{C}^4$, the following conditions are equivalent:
\begin{align*} \label{eq:oredge}
\mathrm{Re}\left(\tfrac{a+b+c+d'}{abc}\right)\geq\mathrm{Re}\left(\tfrac{a+b+c+d}{abc}\right)
\ \ \Leftrightarrow\ \  {\rm Re}\left(\tfrac{d'}{a+b+c+d'}\right)\geq\tfrac{1}{2}\ \ \Leftrightarrow\ \ |d'|\geq |d|.
\end{align*}
\end{lem}
\begin{proof}
The edge relation \eqref{edgerel} proves the first equivalence.  Furthermore \eqref{edgerel} also proves that
${\rm Im}\left(\tfrac{a+b+c+d'}{abc}\right)=-{\rm Im}\left(\tfrac{a+b+c+d}{abc}\right)$ hence the first inequality is equivalent to
\begin{align*}
\left|\frac{a+b+c+d'}{abc}\right|\geq\left|\frac{a+b+c+d}{abc}\right|
\end{align*}
which is equivalent to
\begin{align*}
\left|\frac{(a+b+c+d')^2}{abc}\right|\geq\left|\frac{(a+b+c+d)^2}{abc}\right|.
\end{align*}
By the vertex relation \eqref{vertrel}, this is precisely $|d'|\geq |d|$.
\end{proof}

\subsection{Local analysis}

\begin{dfn}
Call a vertex with all outward pointed oriented edges a \emph{source}, a vertex with all inwardly pointed oriented edges a \emph{sink} and a vertex with precisely one outwardly pointed oriented edge a \emph{funnel}. The remaining two types of vertices are called \emph{saddles}.
\end{dfn}

\begin{lem}  \label{th:nosource}
There are no sources.
\end{lem}

\begin{proof}
Given such a vertex with adjoining $3$-cells $A,B,C,D$, the vertex relation \eqref{vertrel} gives:
\begin{align*}
1=&\mathrm{Re}\left(\frac{a+b+c+d}{abc}+\frac{a+b+c+d}{abd}+\frac{a+b+c+d}{acd}+\frac{a+b+c+d}{bcd}\right)\\
\geq&\frac{1}{2}+\frac{1}{2}+\frac{1}{2}+\frac{1}{2}=2
\end{align*}
which is a contradiction.
\end{proof}

\begin{remark}\label{rmK:cell} In the Teichm\"uller component, Markoff maps take on real positive values and the proof of Lemma~\ref{th:nosource} shows that any vertex has at most one outgoing edge. This means that Fuchsian Markoff maps can have at most one sink (given some choice of orientation). We later show in theorem ~\ref{th:4nonempty} that a sink always exists, and this may be geometrically interpreted as saying that there is a unique (up to permutation) Markoff quad $(a,b,c,d)$ for any 3-cusped projective plane where the flips of $a,b,c$ of $d$ are (non-strictly) longer. However, the Markoff quad $(a,b,c,d)=\frac{i}{\sqrt{2}}(1,1,1,-2)$ is an example of a vertex with three outgoing edges.
\end{remark}

\begin{lem}
At a sink vertex, a Markoff quad contains an element of magnitude less than or equal to 4.
\end{lem}

\begin{proof}
Let the real parts of
\begin{align*}
\frac{a+b+c+d}{abc},\frac{a+b+c+d}{abd},\frac{a+b+c+d}{acd},\frac{a+b+c+d}{bcd},
\end{align*}
respectively be $s\leq r\leq q\leq p$.\newline
\\
Then the sink part tells us that $p\leq \frac{1}{2}$ and the size ordering and the fact that $s+r+q+p=1$ tells us that $p\geq\frac{1}{4}$. And we also know that the next largest number $q\geq\frac{1}{3}(1-p)$. Therefore,
\begin{align*}
pq\geq\frac{1}{3}p(1-p)\geq\frac{1}{3}\times\frac{1}{4}\times\frac{3}{4}=\frac{1}{16}.
\end{align*}
Then, we see that:
\begin{align*}
\frac{1}{|cd|}=|\frac{a+b+c+d}{bcd}|\times|\frac{a+b+c+d}{acd}|\geq pq\geq \frac{1}{16}.
\end{align*}
Therefore, $|cd|\leq 16$ and the lesser of the magnitudes of these two traces must be less than or equal to $4$.
\end{proof}

\begin{lem}\label{th:<2}
Given a saddle vertex $\{\alpha,\beta,\gamma,\delta\}$ with two outgoing oriented edges
\begin{align*}
\{\alpha,\beta,\gamma;\delta'\rightarrow \delta\}\text{ and }\{\alpha,\beta,\delta:\gamma'\rightarrow \gamma\},
\end{align*}
then the 2-cell $\{\alpha,\beta\}$ lies in $\Omega^2_\phi(4)$ and at least one of $\{\alpha\},\{\beta\}$ lies in $\Omega^3_\phi(2)$.
\end{lem}

\begin{proof}
The outwards pointing condition tells us that:
\begin{align*}
\left|\frac{c}{a+b+c+d}\right|\geq\frac{1}{2}\text{ and }\left|\frac{d}{a+b+c+d}\right|\geq\frac{1}{2}.
\end{align*}
Multiplying these terms together, we have:
\[
\left|\frac{cd}{(a+b+c+d)^2}\right|=\frac{1}{|ab|}\geq\frac{1}{4}\Rightarrow|ab|\leq4\Rightarrow\mathrm{min}\{|a|,|b|\}\leq2.
\]
\end{proof}

\subsection{Global analysis}

\begin{lem}  \label{th:conn}
For $k\geq2$, the cell complex comprised of all the $3$-cells in $\Omega_\phi^3(k)$ is connected.
\end{lem}

\begin{proof}
Assume that $\Omega_\phi^3(k)$ isn't connected and consider a shortest path of oriented edges between two distinct connected components:
\begin{align*}
\vec{e}_1,\vec{e}_2,\ldots,\vec{e}_{p-1},\vec{e}_p.
\end{align*}
Note that by assumption, any $3$-cell $X$ that contains one of these edges must satisfy $|\phi(X)|>k$.\newline
\\
If $p=1$, then $\vec{e}_1=\{\alpha,\beta,\gamma;\delta\rightarrow \delta'\}$ such that $|d|,|d'|\leq k$. Then the edge relation \eqref{edgerel} gives $abc=(a+b+c+d)+(a+b+c+d')$ hence:
\begin{align*}
\sqrt{k^3}<\sqrt{|abc|}\leq&\frac{|a+b+c+d|+|a+b+c+d'|}{\sqrt{|abc|}}\\
=&\sqrt{|d|}+\sqrt{|d'|}\leq2\sqrt{k}\\
\Rightarrow k^3&\leq 4k\Rightarrow k\leq 2,
\end{align*}
where the first equality uses the vertex relation \eqref{vertrel}.  This contradicts the assumption.\newline
\\
On the other hand, if $p\geq2$, then $\vec{e}_1$ must point away from $\vec{e}_2$ and $\vec{e}_p$ must point away from $\vec{e}_{p-1}$. But this means that at least one of the interior vertices of the path $\{\vec{e}_n\}_{n=1,\ldots,p}$ must have two arrows pointing away from it, and hence by Lemma~\ref{th:<2} one of the adjacent $3$-cells $X$ of this vertex must satisfy  $|\phi(X)|\leq 2$, thus contradicting the assumption.
\end{proof}

\begin{lem}
\label{ray}
Given an infinite ray of oriented edges $\{\vec{e}_n\}_{n\in{\mathbb{N}}}$ such that each $\vec{e}_n$ is directed towards $\vec{e}_{n+1}$, then this ray either:
\begin{enumerate}
\item
eventually spirals along the boundary of some 2-cell $\{\xi,\eta\}\in\Omega^2_\phi(4)$, or
\item
eventually enters and remains on the boundary of some 3-cell $\{\xi\}\in\Omega^3_\phi(2)$, or
\item
there are infinitely many 3-cells in $\Omega^3_\phi(2)$.
\end{enumerate}

\end{lem}

\begin{proof}
We begin by four-colouring $\Omega^3$ with the colours $\alpha,\beta,\gamma,\delta$. In particular, we label the $3$-cells meeting $\{\vec{e}_n\}$ by $\{\alpha_i\},\{\beta_j\},\{\gamma_k\},\{\delta_l\}$ where the letter type is determined by the colour of the cell and the subscripts grow according to how early we encounter this $3$-cell as we traverse along $\{\vec{e}_n\}$.\newline
\\
At each vertex along $\{\vec{e}_n\}$, we encounter six $2$-cells of different colour-types:
\begin{align*}
\{\alpha\beta,\alpha\gamma,\alpha\delta,\beta\gamma,\beta\delta,\gamma\delta\}.
\end{align*}
Now consider all the $2$-cells that we encounter as we go along $\{\vec{e}_n\}$. Since $\{\vec{e}_n\}$ does not repeat its edges, if we ever meet only finitely many $2$-cells of a certain colour-type, then $\{\vec{e}_n\}$ eventually just stays on the last $2$-cell we meet of that cell-type. This also means that we can't meet only finitely many $2$-cells of two different colour-types because we must then stay on two distinct $2$-cells - impossible because the intersection of any two $2$-cells is either empty or consists of a single edge.\newline
\\
Assume that we encounter only finitely many $2$-cells of the (wlog) $\alpha\beta$ colour-type and that the $2$-cell that we stay on is (with a little notation abuse) $\{\alpha,\beta\}$. This means that the vertices of $\{\vec{e}_n\}$ eventually take the following form:
\begin{align*}
\ldots,\{\alpha,\beta,\gamma_i,\delta_j\},\{\alpha,\beta,\gamma_{i+1},\delta_j\},\{\alpha,\beta,\gamma_{i+1},\delta_{j+1}\},\{\alpha,\beta,\gamma_{i+2},\delta_{j+1}\},\ldots
\end{align*}
and the sequences $\{|c_i|\},\{|d_j|\}$ are (monotonically) non-increasing due to the directions of the oriented edges. Alternatively, we phrase this as the statements that:
\begin{align*}
\mathrm{Re}\left(\frac{c_i}{a+b+c_i+d_j}\right)\geq\frac{1}{2}\text{ and } \mathrm{Re}\left(\frac{d_j}{a+b+c_{i+1}+d_j}\right)\geq\frac{1}{2},
\end{align*}
noting that the latter statement implies that:
\begin{align*}
2\sqrt{|d_j|}\geq\frac{|a+b+c_{i+1}+d_j|}{\sqrt{|d_j|}}=\sqrt{|abc_{i+1}|}.
\end{align*}
 Now, if the sequence $\{|c_i|\}$ is bounded below by $2$, it must converge. Thus for any $\epsilon>0$, by choosing $i$ to be sufficiently large, $\sqrt{|c_{i+1}|}\leq\sqrt{|c_i|}\leq\sqrt{|c_{i+1}|}+\epsilon$. Then the edge relation \eqref{edgerel} for $\{\alpha,\beta,\delta_j;\gamma_i\rightarrow \gamma_{i+1}\}$ tells us that:
\begin{align*}
\sqrt{|abd_j|}&\leq\frac{|a+b+c_i+d_j|}{\sqrt{|abd_j|}}+\frac{|a+b+c_{i+1}+d_j|}{\sqrt{|abd_j|}}\\
&=\sqrt{|c_i|}+\sqrt{|c_{i+1}|}\leq2\sqrt{|c_{i+1}|}+\epsilon.
 \end{align*} 
Combining this with the inequality above, we see that:
\begin{align*}
|ab|\leq\frac{4\sqrt{|c_{i+1}|}+2\epsilon}{\sqrt{|c_{i+1}|}}\leq4+\sqrt{2}\epsilon.
\end{align*}
Therefore, $|ab|\leq4$.\newline
\\
We have now covered the case where we meet only finitely many 2-cells of one of the colour-types. The alternative is that we meet infinitely many 2-cells of all six colour-types, and we produce from this four sequences of 3-cells:
\begin{align*}
\left\{\{\alpha_i\}\},\{\{\beta_j\}\},\{\{\gamma_k\}\},\{\{\delta_l\}\right\}
\end{align*}
Now, the second case arises when one of these sequences is finite --- that is, we stick to the surface of some 3-cell. Assume wlog that this is for the colour $\alpha$, and by truncating our ray (and abusing notation), we may take $a_j=a=\phi(\{\alpha\})$ for all $j$. Moreover, unless we're in case $3$, we may further truncate our ray so that the non-increasing sequence $\{|b_j|\},\{|c_k|\},\{|d_l|\}$ remains bounded above $2$. Then the same analysis tells us that:
\begin{align}
|ab_i|,|ac_j|,|ac_k|\rightarrow 4,
\end{align}
and we can see from this that $|a|\leq2$.\newline
\\
Finally, in the case that we meet infinitely many $3$-cells of every colour-type, assume that the monotonically non-increasing sequences $\{|a_i|\},\{|b_j|\},\{|c_k|\},\{|d_l|\}$ are bounded below by $2$ and hence converge. The same analysis as in case one tells us that
\begin{align*}
|a_i b_j|,|a_i c_k|, |a_i d_l|, |b_j c_k|, |b_j d_l|, |c_k d_l|\rightarrow 4,
\end{align*}
and since these numbers are the bound below by 2, we see that:
\begin{align*}
|a_i|,|b_j|,|c_k|,|d_l|\rightarrow2.
\end{align*}

Now, for the oriented edge $\{\alpha,\beta,\gamma;\delta\rightarrow \delta'\}$ sufficiently far along $\{\vec{e}_n\}$ so that $|a|,|b|,|c|,|d|,|d'|$ are each close to $2$, the edge relation \eqref{edgerel}
\begin{align*}
\frac{a+b+c+d}{abc}+\frac{a+b+c+d'}{abc}=1
\end{align*}
tells us that:
\begin{align*}
\frac{a+b+c+d}{abc},\frac{a+b+c+d'}{abc}\approx\frac{1}{2}
\end{align*}
By symmetry, this also holds for:
\begin{align*}
\frac{a+b+c+d}{abc},\frac{a+b+c+d}{abd},\frac{a+b+c+d}{acd},\frac{a+b+c+d}{bcd}\approx\frac{1}{2}.
\end{align*}
By mutliplying pairs of these terms and invoking the vertex relation \eqref{vertrel}, we obtain that:
\begin{align*}
ab,ac,ad,bc,bd,cd\approx 4,
\end{align*}
and hence either $a,b,c,d$ are approximately all $2$ or all $-2$. But the vertex relation \eqref{vertrel} then tells us that
\begin{align*}
64\approx(a+b+c+d)^2\approx abcd\approx 16,
\end{align*}
giving us the desired contradiction for our assumption that these sequences could be bounded below by $2$.\newline
\\
In particular, this shows us that we must touch some 3-cell in $\Omega^3_\phi(2)$, and the subsequent infinitely many 3-cells of the same colour as $X$ must all be in $\Omega^3_\phi(2)$.
\end{proof}

\begin{thm}  \label{th:4nonempty}
The set of 3-cells $\Omega^3_\phi(4)$ is non-empty. Further, if $\Omega^3_\phi(2)=\varnothing$, then there is a unique sink.
\end{thm}

\begin{proof}
If $\Omega^3_\phi(2)$ is non-empty then we're done. But if it is empty, then lemma~\ref{ray} tells us that following oriented edges according to their directions must eventually result in a sink. If there are multiple sinks, they obviously cannot be distance 1 from each other. And one of the interior vertices of any path joining two sinks must have two arrows coming out of it and hence by lemma~\ref{th:<2}, the set $\Omega^3_\phi(2)$ is non-empty.
\end{proof}

\subsection{Systolic inequality.}

\begin{repeatthm} [Systolic inequality]
Let $\rho$ denote a quasi-Fuchsian representation for a thrice-punctured projective plane, then
\begin{align}
\sys(X_\rho)\leq2\arcsinh(2).
\end{align}
In particular, the unique maximum of the systole function over the moduli space of all hyperbolic thrice-punctured projective planes is $2\arcsinh(2)$.
\end{repeatthm}

\begin{proof}
Any quasi-Fuchsian representation $\rho$ induces a BQ-Markoff map $\phi$. By Theorem~\ref{th:4nonempty}, $\Omega^3_\phi(4)$ is non-empty: on the hyperbolic manifold $X_\rho$, there exists a one-sided simple geodesic $\gamma$ with $|\tr A|=|2\sinh\frac{1}{2}\ell_{\gamma}(X)|\leq 4$ and hence $\ell_{\gamma}(X)\leq 2\hspace{.4mm} \arcsinh(2)$.  Thus, the maximum of the systole length function over the set of BQ-Markoff maps is less than or equal to $2\hspace{.4mm} \arcsinh(2)$.\newline
\newline
To prove equality, consider the Markoff quad $(4,4,4,4)$, which we know from lemma~\ref{lem:markoff} arises from a Fuchsian representation. Any new Markoff quad generated from $(4,4,4,4)$ must be integral, and each entry is a positive multiple of $4$. Thus, the corresponding Markoff map has $4$ as its minimum. This in turn means that the shortest one-sided geodesic has length $2\arcsinh(2)$. On the other hand, the shortest two-sided geodesic has trace $14=4\times4-2$, is of length $2\arccosh(7)>2\arcsinh(2)$ and hence cannot be a systolic geodesic.\newline
\newline
To prove the uniqueness of the maximum of the systole function over the moduli space $\mathcal{M}(N_{1,3})$, first recall from remark~\ref{rmK:cell} that for any 3-cusped projective plane, there exists a positive real Markoff quad $(a,b,c,d)\in\mathbb{R}_+^4$ such that
\begin{align*}
\frac{a}{a+b+c+d},\frac{b}{a+b+c+d},\frac{c}{a+b+c+d},\frac{d}{a+b+c+d}\leq\frac{1}{2}.
\end{align*}
If $(a,b,c,d)\neq(4,4,4,4)$ is a maximum of the systole function, we assume wlog that $4=a\leq b\leq c\leq d$ and $4<d$. Define
\begin{align*}
0\leq x_b=b-a\leq x_c=c-a\leq x_d=d-a\text{ and }0<x_d.
\end{align*}
Expanding equation~\eqref{quad} in terms of these new quantities, we have:
\begin{align}\label{eq:ineq}
x_b^2+x_c^2+x_d^2=32(x_d+x_c+x_d)+14(x_bx_c+x_bx_d+x_cx_d)+4x_bx_cx_d.
\end{align}
Since $x_b^2\leq x_bx_c$ and $x_c^2\leq x_cx_d$, we obtain from equation~\eqref{eq:ineq} that $x_d^2\geq (32+13x_c)x_d$. Therefore:
\begin{align*}
d\geq 36+13x_c\geq 9\max\{c,4\}>a+b+c.
\end{align*}
But this contradicts the fact that $\frac{d}{a+b+c+d}\leq\frac{1}{2}$.
\end{proof}

\begin{remark}
We can recognise the minimum of the systole geometrically due to its large symmetry group.  Consider the spherical symmetric octahedron: the octahedron on the round two-sphere $S^2$ with great circle edges and full $A_4$ symmetry.  Label the 6 vertices of this octahedron to get a symmetric element $\Sigma$ in the moduli space $\mathcal{M}_{0,6}$, and note that the 6 labeled points are invariant under the antipodal map.  There exists a unique hyperbolic cusped surface $X$ with conformal structure $\Sigma$. And by the uniqueness of $X$, the vertex-fixing antipodal maps on $S^2$ uniformise to isometric $\mathbb{Z}_2$-actions on $X$. The 4 greater circles on $S^2$ which lie in the plane orthogonal to the vector between the centers of any two opposing faces uniformise to simple closed geodesics $\gamma_1,\gamma_2,\gamma_3,\gamma_4\in X$.  By symmetry, each $\gamma_i$ is invariant under antipodal $\mathbb{Z}_2$-actions and descends to a geodesic $\overline{\gamma}_i\in X/\mathbb{Z}_2$, where $X/\mathbb{Z}_2$ is the desired 3-cusped projective plane.  By symmetry, the four  geodesics $\overline{\gamma}_i$ in $X/\mathbb{Z}_2$ have the same length, and hence their traces give rise to a Markoff quad $(\ell,\ell,\ell,\ell)$ which is necessarily $(4,4,4,4)$.  
\end{remark}

\section{Fibonacci Growth}  \label{sec:fib}
For any Markoff quad, solve for $d$ to get
\begin{align*}
d=\frac{abc}{4}\left(1\pm\sqrt{1-4\left(\frac{1}{ab}+\frac{1}{ac}+\frac{1}{bc}\right)}\,\right)^2.
\end{align*}
For $|a|$, $|b|$, $|c|$ large, choose $d$ to be the larger of the two solutions, hence
\begin{align*}
\log|d|\approx \log|a|+\log|b|+\log|c|.
\end{align*} 
In particular, $|d|$ is greater than $|a|$, $|b|$ and $|c|$. If we continue and flip from $a$ to $a'$,  then $\log|a|<\log|a'|\approx \log|b|+\log|c|+\log|d|$.  This gives rise to the notion of \emph{Fibonacci growth} for Markoff quads, in keeping with Bowditch's Fibonacci growth for Markoff triples \cite{BowMar}.\newline
\newline
The goal of this section is to define and establish the Fibonacci growth for BQ-Markoff maps defined in \eqref{BQ}, and use these growth rates to prove McShane identities and length spectrum growth rates. It should be noted that the Fibonacci growth rates that we define and prove here are strictly stronger than similar growth rates found in \cite{HuHIde}, although only the weaker version is needed to prove theorem~\ref{th:main}.

\subsection{Fibonacci growth.} Given an edge $e\in\Omega^1$, define the \emph{Fibonacci function} $F_e:\Omega^3\rightarrow\mathbb{R}$ by:  
\begin{itemize}
\item[(i)] $F_e(\alpha)=1$ if $e=\{\alpha,\beta,\gamma\}$.
\item[(ii)] For $\{\alpha,\beta,\gamma,\delta\to\delta'\}\in\vec{\Omega}^1$ oriented so that it points away from $e$ (or is either of the two possible oriented edges for $e$ itself)
\begin{align*}
F_e(\{\delta\})=F_e(\{\alpha\})+F_e(\{\beta\})+F_e(\{\gamma\}).
\end{align*}
\end{itemize}
Hence $F_e:\Omega^3\rightarrow\mathbb{R}$ takes the value of $1$ for the three $3$-cells in $\Omega^3$ that contain $e$ and subsequently define values for the rest of the tree by assigning to every hitherto unassigned $3$-cell meeting three assigned $3$-cells at some vertex the sum of the values of those already assigned $3$-cells.

\begin{dfn}  \label{fibound}
Given a function $f:\Omega^3\rightarrow[0,\infty)$ and $\Omega'\subset\Omega^3$, we say that $f$ has:
\begin{itemize}
\item
a \emph{lower Fibonacci bound} on $\Omega'$ if there's some positive $\kappa$ such that:
\begin{align*}
\frac{1}{\kappa} F_e(X)\leq f(X)\text{ for all but finitely many }X\in\Omega';
\end{align*}
\item
an \emph{upper Fibonacci bound} on $\Omega'$ if there's some positive $\kappa$ such that:
\begin{align*}	
f(X)\leq \kappa F_e(X)\text{ for all }X\in\Omega';
\end{align*}
\item
\emph{Fibonacci growth} on $\Omega'$ if there's some positive $\kappa$ such that:
\begin{align*}
\frac{1}{\kappa} F_e(X)\leq f(X)\leq \kappa F_e(X)\text{ for all but finitely many }X\in\Omega';
\end{align*}
or in other words: it has both lower and upper Fibonacci bound. We also opt to omit \emph{``on $\Omega'$''} whenever $\Omega'=\Omega^3$.
\end{itemize}
\end{dfn}
We assumed the choice of an edge $e$ for these definitions, and now show that the existence of a $\kappa$ satisfying these conditions is independent of this choice.

\begin{lem}\label{upper}
Given some edge $e$ that is the intersection of the three $3$-cells $X_1,X_2,X_3$ and a function $f:\Omega^3\rightarrow[0,\infty)$ satisfying:
\begin{align*}
f(D)\leq f(A)+f(B)+f(C)+2c,\,0\leq c,
\end{align*}
where $A,B,C,D$ meet at the same vertex and $D$ is strictly farther from $e$ than $A,B,C$. Then:
\begin{align*}
f(X)\leq(M+c)F_e(X)-c,\text{ for all }X\in\Omega^3,
\end{align*}
where $M=\mathrm{max}\{f(X_1),f(X_2),f(X_3)\}$.
\end{lem}

\begin{proof}
We prove this by induction on the distance of a region from $e$. The base case is due to:
\begin{align*}
f(X_i)\leq(\mathrm{max}\{f(X_1),f(X_2),f(X_3)\}+c)-c.
\end{align*}
The induction step is similarly established:
\[
f(D)\leq(M+c)(F_e(A)+F_e(B)+F_e(C))-3c+2c=(M+c)F_e(D)-c.
\]
\end{proof}

Note that by essentially the same proof, we obtain the following result:

\begin{lem}
Given some edge $e$ that is the intersection of the three $3$-cells $X_1,X_2,X_3$ and a function $f:\Omega^3\rightarrow[0,\infty)$ satisfying:
\begin{align*}
f(D)\geq f(A)+f(B)+f(C)-2c,\,0\leq c<m:=\mathrm{min}\{f(X_1),f(X_2),f(X_3)\},
\end{align*}
where $A,B,C,D$ meet at the same vertex and $D$ is strictly farther from $e$ than $A,B,C$. Then:
\begin{align*}
f(X)\geq(m-c)F_e(X)+c,\text{ for all }X\in\Omega^3.
\end{align*}
\end{lem}

However, this is insufficient for our purposes. We shall require:

\begin{lem}\label{lower}
Given some oriented edge $\vec{e}$ that is the intersection of the three $3$-cells $X_1,X_2,X_3$ and a function $f:\Omega^3\rightarrow[0,\infty)$ satisfying:
\begin{align*}
f(D)\geq f(A)+f(B)+f(C)-2c,\,0\leq c<\mu:=\mathrm{min}\{f(X_i)+f(X_j)\}_{i\neq j},
\end{align*}
where $A,B,C,D$ meet at the same vertex and $D$ is strictly farther from $e$ than $A,B,C$. Then:
\begin{align*}
f(X)\geq(\mu-2c)F_e(X)+c,\text{ for all }X\in\Omega^3_-(\vec{e})-\Omega^3_0(\vec{e}).
\end{align*}
\end{lem}

\begin{proof}
We first use induction to show that any two adjacent $3$-cells in $\Omega^3_-(\vec{e})$ satisfy:
\begin{align*}
f(X)+f(Y)\geq (\mu-2c)(F_e(X)+F_e(Y))+2c.
\end{align*}
The base case where $X$ and $Y$ are both in $\Omega_0(e)$ follows from the definition of $\mu$. We proceed by induction on the total distance of $X$ and $Y$ from $e$. Assume that $Y$ is farther than $X$ from $e$. The tree structure of $\Omega^3_-(\vec{e})$ means that there is a unique closest vertex between the edge $e$ and the face $X\cap Y$. Denote the two other $3$-cells at this vertex by $W$ and $Z$, we then have:
\begin{align*}
f(X)+f(Y)\geq& f(X)+f(W)+f(X)+f(Z)-2c\\
\geq&(\mu-2c)(F_e(X)+F_e(W)+F_e(X)+F_e(Z))+4c-2c\\
=&(\mu-2c)(F_e(X)+F_e(Y))+2c,
\end{align*}
completing the induction.\newline
\newline
Now consider a $3$-cell $D\in\Omega^3_-(\vec{e})-\Omega^3_0(e)$, and denote by $A,B,C$ the three other $3$-cells meeting $D$ at the closest vertex between $e$ and $D$. Then we have:
\begin{align*}
f(D)\geq&\frac{1}{2}(f(A)+f(B)+f(C)+f(D)-2c)\\
\geq&\frac{1}{2}(\mu-2c)(F_e(A)+F_e(B)+F_e(C)+F_e(D))+2c-c\\
=&(\mu-2c)F_e(D)+c.
\end{align*}
\end{proof}

Since for any edge $e'$, the function $F_{e'}$ satisfies the criteria for these last two lemmas, we see that there is some $\kappa>0$ such that:
\begin{align*}
\frac{1}{\kappa}F_e(X)\leq F_{e'}(X)\leq\kappa F_e(X),\text{ for all }X\in\Omega^3.
\end{align*}
Which shows that Definition~\ref{fibound} is indeed independent of the choice of the edge $e$.

\begin{lem}   \label{th:sumcon}
If a function $f:\Omega^3\rightarrow\mathbb{R}^+$ has a lower Fibonacci bound, then for any $\sigma>3$, the following sum converges:
\begin{align*}
\sum_{X\in\Omega}f(X)^{-\sigma}<\infty.
\end{align*}
\end{lem}
\begin{proof} It suffices for us to show that this sum converges for $f= F_e$.  We do this by bounding the growth of the level sets of $F_e$. We will prove that:
\begin{equation}  \label{fibtot}
\Card\left\{\,X\in\Omega^3\mid\ F_e(X)=n\right\}<4J_2(n)
\end{equation}
where $J_k$ is the Jordan totient function.  Hence
\begin{align*}
\sum_{X\in\Omega}F_e(X)^{-\sigma}<\sum_{n\geq1}4J_2(n)n^{-\sigma}=\frac{4\zeta(\sigma-2)}{\zeta(\sigma)}
\end{align*}
for $\zeta$ the Riemann zeta function,
and the sum converges for $\sigma>3$.\newline
\newline
For the remainder of this proof, we think of $F_e$ not just as a function on the $3$-cells $\Omega^3$, but also as a set-valued function on the 1-cells, where it assigns to each edge $\{\alpha,\beta,\gamma\}\in\Omega^1$ the unordered 3-tuple $\{F_e(\{\alpha\}),F_e(\{\beta\}),F_e(\{\gamma\})\}$.\newline
\newline
When $n>1$, there is a $1:3$ correspondence between
\begin{align*}
\{X\in\Omega^3\mid F_e(X)=n\}\text{ and }\left\{\{\alpha,\beta,\gamma\}\in\Omega^1\mid \max F_e(\{\alpha,\beta,\gamma\}) =n\right\}
\end{align*}
defined by assigning to $X\in\Omega^3$ the three edges closest to $e$ that lie on $X$.  By uniqueness up to symmetry of values of $F_e$ on paths in $\Omega$, the preimage of any unordered triple $\{l,m,n\}$ in the image of $F_e$ has cardinality at most:
\begin{itemize}
\item 1, if $\{l,m,n\}=\{1,1,1\}$,
\item 6, if $\{l,m,n\}=\{1,1,n\}$ and
\item 12, if $\{l,m,n\}$ are all distinct integers.
\end{itemize}
Thus, the relation given by assigning to a 3-cell $X$ the unordered 3-tuples of values of $F_e$ on the edges on $X$ closest to $e$ is at most $4:1$.  Any triple $\{l,m,n\}$ that is in the image of $F_e$ must be relatively prime. Otherwise, a common factor would inductively propagate back to $e$ and contradict the starting value of $\{1,1,1\}$. Thus, for $n>1$,
\begin{align*}
\Card\left\{X\in\Omega^3\mid F_e(X)=n\right\}&\leq 4\,\Card\left\{\, \{l,m,n\} \mid l,m<n\text{ and }\gcd(l,m,n)=1\right\}\\
&<4\,\Card\left\{\, (l,m,n) \mid l,m\leq n\text{ and }\gcd(l,m,n)=1\right\}\\
&=4J_2(n)\text{, and \eqref{fibtot} holds as required.}
\end{align*}
\end{proof}

These results enable us to conclude that: if the function
\begin{align*}
\log^+|\phi|:\Omega^3\rightarrow[0,\infty)
\end{align*}
satisfies the following inequality at every vertex $\{a,b,c,d\}\in\Omega^0$:
\begin{align}
\log^+|d|\leq\log^+|a|+\log^+|b|+\log^+|c|+2\log\left(\frac{1+\sqrt{13}}{2}\right),\label{ineq}
\end{align}
where $\log^+(x):=\mathrm{max}\{0,\log(x)\}$, then:
\begin{lem}
$\log^+|\phi|$ has an upper Fibonacci bound on $\Omega^3$.
\end{lem}

\begin{proof}
By the preceding comment, we only need to show that \eqref{ineq} holds. To begin with, we see that when $|d|\leq1$, the desired identity is trivially satisfied. We therefore confine ourselves to when $|d|>1$, that is: when $\log|d|=\log^+|d|$. We now assume without loss of generality that $|a|\leq|b|\leq|c|$ and case-bash the desired result.

\begin{enumerate}
\item
If $1\leq|a|,|b|,|c|$, then:
\begin{align*}
\log^+|d|=\log|d|=&\log|abc|+2\log\left|\frac{1}{2}\left(1\pm\sqrt{1-4(\frac{1}{ab}+\frac{1}{ac}+\frac{1}{bc})}\right)\right| \\
\leq&\log|abc|+2\log\left|\frac{1}{2}\left(1+\sqrt{1+4(\frac{1}{|ab|}+\frac{1}{|ac|}+\frac{1}{|bc|})}\right)\right|\\
\leq&\log|a|+\log|b|+\log|c|+2\log\left(\frac{1+\sqrt{13}}{2}\right)
\end{align*}

\item
If $|a|<1\leq|b|,|c|$, then:
\begin{align*}
\log^+|d|=&\log|bc|+2\log\left|\frac{1}{2}\left(\sqrt{|a|}\pm\sqrt{1-4(\frac{|a|}{ab}+\frac{|a|}{ac}+\frac{|a|}{bc})}\right)\right| \\
\leq&\log|bc|+2\log\left|\frac{1}{2}\left(\sqrt{|a|}+\sqrt{1+4(\frac{1}{|b|}+\frac{1}{|c|}+\frac{|a|}{|bc|})}\right)\right|\\
\leq&\log|b|+\log|c|+2\log\left(\frac{1+\sqrt{13}}{2}\right)
\end{align*}

\item
And similarly, if $|a|,|b|<1\leq|c|$, then:
\begin{align*}
\log^+|d|=&\log|c|+2\log\left|\frac{1}{2}\left(\sqrt{|ab|}\pm\sqrt{1-4(\frac{|ab|}{ab}+\frac{|ab|}{ac}+\frac{|ab|}{bc})}\right)\right| \\
\leq&\log|c|+2\log\left|\frac{1}{2}\left(\sqrt{|ab|}+\sqrt{1+4(1+\frac{|b|}{|c|}+\frac{|a|}{|c|})}\right)\right|\\
\leq&\log|c|+2\log\left(\frac{1+\sqrt{13}}{2}\right)
\end{align*}

\item
And finally, if $|a|,|b|,|c|<1$, then:
\begin{align*}
\log^+|d|=&2\log\left|\frac{1}{2}\left(\sqrt{|abc|}\pm\sqrt{1-4(\frac{|abc|}{ab}+\frac{|abc|}{ac}+\frac{|abc|}{bc})}\right)\right| \\
\leq&2\log\left(\frac{\sqrt{|ab|}+\sqrt{1+4(|a|+|b|+|c|)}}{2}\right)\leq2\log\left(\frac{1+\sqrt{13}}{2}\right)
\end{align*}
\end{enumerate}
\end{proof}

We now aim to show that $\log|\phi|$ has a lower Fibonacci bound, and introduce the following notation: given an oriented edge $\vec{e}$ on $\Omega$, the removal of the edge $e$ from the tree in $\Omega$ results in two connected components. We denote the collection of $3$-cells containing edges from the tree on the \emph{head-side} of $\vec{e}$ by
\begin{align*}
\Omega^3_+(\vec{e})\text{; and }\Omega^3_-(\vec{e})
\end{align*}
for the the collection of $3$-cells containing edges from the tree on the \emph{tail-side} of $\vec{e}$. We also use the notation $\Omega_0^3(e)=\Omega^3_+(\vec{e})\cap\Omega^3_-(\vec{e})$ to refer to the three edges containing $e$.

\begin{lem}
Given an oriented edge $\vec{e}\in\vec{\Omega}^1$ such that $\Omega^3_0(e)\cap\Omega^3_\phi(2)=\varnothing$, then $\Omega^3_\phi(2)$ lies on the head-side of $e$, that is:
\begin{align*}
\Omega^3_\phi(2)\subseteq\Omega^3_+(\vec{e}).
\end{align*}
Furthermore, all oriented edges in $\Omega^3_-(\vec{e})$ must point toward $e$.
\end{lem}

\begin{proof}
Due to the connectedness of $\Omega^3_\phi(2)$, it must lie in either $\Omega^3_+(\vec{e})$ or $\Omega^3_-(\vec{e})$. If it lies on the tail side of $\vec{e}$ then consider a shortest path containing $\vec{e}$ and touching $\Omega^3_\phi(2)$. No $3$-cell in $\Omega^3_\phi(2)$ may be in direct contact with $\vec{e}$ as this would force the region on the other end of $\vec{e}$ be in $\Omega^3_\phi(2)$ --- yielding a contradiction.\newline
\\
Hence, we have a path of length at least $2$ with outwardly oriented edges at the two end of this path. Thus resulting in at least one internal vertex with two outward pointing edges and hence an adjacent result in $\Omega^3_\phi(2)$. This then contradicts the shortest assumption we placed on our path.\newline
\\
We have shown that $\Omega^3_\phi(2)$ is on the head-side of $\vec{e}$ and by lemma~\ref{th:<2}, every vertex on the tail-side of $\vec{e}$ must have three incoming edges and one outgoing edge. Then lemma~\ref{ray} forces all of these edges to point toward $e$.
\end{proof}

\begin{lem}\label{tighter}
Given the hypotheses of the above result, define:
\begin{align*}
\mu:=\mathrm{min}\left\{\log^+|\phi(X_i)|+\log^+|\phi(X_j)|\mid X_k\in\Omega_0^3(\vec{e})\right\}>2\log2,
\end{align*}
then for every tail-side 3-cell $X\in\Omega_-^3(\vec{e})-\Omega^3_0(e)$, we have:
\begin{align*}
\log^+|\phi(X)|\geq(\mu-2\log2)F_e(X)+\log2,
\end{align*}
and hence $\log^+|\phi|$ has a lower Fibonacci bound over $\Omega^3_-(\vec{e})$.
\end{lem}

\begin{proof}
Let $\{\alpha\},\{\beta\},\{\gamma\},\{\delta\}\in\Omega_-^3(\vec{e})$ be the adjacent 3-cells to an arbitrarily chosen tail-side vertex, such that $\{d\}$ is farthest from $e$. Then we know from every edge being naturally directed towards $e$ that:
\begin{align*}
\sqrt{\frac{|d|}{|abc|}}\geq\frac{1}{2}\Rightarrow \log|d|\geq \log|a|+\log|b|+\log|c|-2\log2.
\end{align*}
Since this is satisfied for every tail-side vertex, lemma~\ref{lower} then gives the desired conclusion.
\end{proof}

Fix an arbitrary $2$-cell $\{\alpha,\beta\}\in\Omega^2$, its boundary is a bi-infinite path which we label by $\{e_n\}_{n\in\mathbb{Z}}$. Each edge $e_n$ is the intersection of three distinct 3-cells, two of which are $\{\alpha\},\{\beta\}$ and the last we'll denote by $\{\gamma_n\}$.

\begin{lem}\label{growth}
Given the above setup, then:
\begin{enumerate}
\item
If $ab\notin [0,4]$, then either $|c_n|$ grows exponentially as $n\rightarrow\pm\infty$ or $c_n=0$.
\item
If $ab\in [0,4)$, then $|c_n|$ remains bounded.
\item
If $ab=4$, then $c_n=A+Bn-(a+b)n^2$ for some $A,B\in\mathbb{C}$.
\end{enumerate}
\end{lem}

\begin{proof}
The edge relation~\eqref{edgerel} then tells us that:
\begin{align*}
c_{n+1}+(2-ab)c_n+c_{n-1}+2(a+b)=0.
\end{align*}

If $ab\neq 0,4$, we may solve for this difference equation:
\begin{align}\label{eq:difference}
c_n=A\lambda^n+B\lambda^{-n}-\frac{2(a+b)}{4-ab}\text{, where }\lambda^{\pm1}=\frac{1}{2}(ab-2\pm\sqrt{ab(ab-4)}).
\end{align}
\textbf{Case (1):} Assume that $ab\notin [0,4]$.  Since $ab=(\lambda^{\frac{1}{2}}+\lambda^{-\frac{1}{2}})^2$, then the fact that $|\lambda|=1$ if and only $ab\in[0,4]$ means that $|\lambda|\neq1$. Thus, to show that $|c_n|$ grows exponentially, it suffices to show that neither $A$ or $B$ equals $0$. We prove this by contradiction: assume wlog that $B=0$ and that $|\lambda|>1$, then
\begin{align*}
c_n=A\lambda^n-\frac{2(a+b)}{4-ab}.
\end{align*}
Substituting this into equation~\eqref{quad} and taking the limit as $n\rightarrow-\infty$, we obtain that:
\begin{align*}
\left(\frac{ab(a+b)}{4-ab}\right)^2=4ab\left(\frac{a+b}{4-ab}\right)^2.
\end{align*}
By assumption, $ab\neq 0,4$ and therefore $a+b=0$. This in turn means that $c_n=A\lambda^n$. Substuting this into equation~\eqref{quad} shows that either $A=0$ or $|\lambda|=1$. Since $\lambda|>1$, we conclude that $c_n=0$ for all $n$.\newline
\newline
\textbf{Case (2):} If $ab\in(0,4)$, then $|\lambda|=1$ and by equation~\eqref{eq:difference}, $|c_n|$ is bounded above by $|A|+|B|+\left|\frac{2(a+b)}{4-ab}\right|$. If $ab=0$, we assume wlog that $b=0$. Then by the vertex relation \eqref{vertrel}, $c_n+c_{n-1}+a=0$ and the sequence is either constant or oscillates between two values. \newline
\newline
\textbf{Case (3):} For $ab=4$, simply solving for the edge relation~\eqref{edgerel} as a difference equation yields the desired expression for $c_n$. 
\end{proof}

\begin{thm}\label{converge}
If $\phi\in\Phi_{BQ}$, then $\log^+|\phi|$ has Fibonacci growth.
\end{thm}

\begin{proof}
If $\Omega^2_\phi(4)=\varnothing$, then there is a unique sink in $\Omega^0$. Otherwise, a path between two sinks would contain some vertex with at least two outward pointing oriented edges and hence be adjacent to a 2-cell in $\Omega^2_\phi(4)$ by lemma~\ref{th:<2}. Then apply lemma~\ref{tighter} to the four oriented edges pointing into this unique sink to obtain the desired Fibonacci lower bound.\newline
\\
Otherwise,
\begin{align*}
\Omega_\phi^2(4)=\left\{\{\alpha_1,\beta_1\},\{\alpha_2,\beta_2\},\ldots,\{\alpha_l,\beta_l\}\right\}
\end{align*}
is finite but non-empty. Then let $T$ denote the smallest tree in $\Omega$ containing the boundaries of all the 2-cells in $\Omega^2_\phi(4)$. We claim that $T$ must contain every sink and saddle. Firstly, it's clear from lemma~\ref{th:<2} that every saddle lies on the boundary of some 2-cell in $\Omega^2_\phi(4)$ and hence in $T$.\newline
\\
Now take an arbitrary sink $v$. Since $\Omega^2_\phi(4)$ is non-empty, $\Omega^3_\phi(2)$ must also be non-empty. Consider the shortest path between $\Omega^3_\phi(2)$ and $v$, if the length of this path is 2 or more, then we reach a contradiction because there must be an internal vertex adjacent to a 2-cell in $\Omega^2_\phi(4)$ and hence a 3-cell in $\Omega^3_\phi(2)$. And if the length of this path is 1, then we contradict the connectedness of $\Omega^3(2)$. Hence $v$ lies on the boundary of some 3-cell $A\in\Omega^3_\phi(2)$.\newline
\\
Now, thanks to the connectedness of $\Omega^3_\phi(2)$, the boundary of $A$ must contain some 2-cell in $\Omega^2_\phi(4)$. Thus the shortest path from $v$ to $\Omega^2_\phi(4)$ lies on the boundary of $A$. Note that the 3-cell at the tail of the chosen edge on this path closest to $\Omega^2_\phi(4)$ must point towards $\Omega^2_\phi(4)$ or else produce a closer 2-cell in $\Omega^2_\phi(4)$. Hence, by similar arguments as used in the previous paragraph, $v$ must lie on the boundary of a 2-cell in $\Omega^2_\phi(4)$. \newline
\\
We now show that all but finitely many vertices are funnels.  Observe that all but finitely many edges in $T$ lie on the boundary of some 2-cell in $\Omega^2_\phi(4)$. Then lemma~\ref{growth} tells us that since $|\phi|$ grows exponentially as we traverse the boundary of a 2-cell, there can only be finitely many sinks. Further observe that by lemma~\ref{ray}, every oriented edge outside of $T$ must point into $T$. This means that along the boundary of any of the 2-cells in $\Omega^2_\phi(4)$, there must  (in all but finitely many cases) be a sink in between two saddles. Hence, we see that the number of saddles is also finite.\newline
\\
We now show that a Fibonacci lower bound holds over the set:
\begin{align*}
\Omega^3_0(T):=\{\text{ 3-cells touching }T\}.
\end{align*}
We know that all but finitely many of the 3-cells in $\Omega^3_0(T)$ spiral around some 2-cell in $\Omega^2_\phi(4)$. And lemma~\ref{growth} tells us that $\log^+|\phi|$ over each of these spirals grows linearly, and hence for the spiral $\Omega^3_0(\{\alpha_i,\beta_i\})-\{\{\alpha_i\},\{\beta_i\}\}$ around $\{\alpha_i,\beta_i\}\in\Omega^2_\phi(4)$, we have:
\begin{align*}
\log^+|\phi(X)|\geq\kappa_i F_e(X)+\mu_i,
\end{align*}
where $\kappa_i$ is a function of $|ab|$ and the minimum of $F_e$ on this spiral around $\{\alpha_i,\beta_i\}$, and $\mu_i$ may be negative. Since there are finitely many such spirals, only finitely many 3-cells in $\Omega^3_0(T)$ not on a spiral and the minimum of $\log^+|\phi|$ is greater than $0$, we see that:
\begin{align*}
\log^+|\phi(X)|\geq\kappa F_e(X),\text{ for all }X\in\Omega^3_0(T).
\end{align*}

Finally, label all the oriented edges touching but not contained in $T$ by $\{\vec{\epsilon}_i\}$ (in order of increasing distance from $e$ if you so wish), and for each $\vec{\epsilon}_i$, label the the three 3-cells in $\Omega^3_0(\epsilon_i)$ by:
\begin{align*}
\Omega^3_0(\epsilon_i)=\{X_i,Y_i,Z_i\},\text{ such that }\log^+|\phi(X_i)|\leq\log^+|\phi(Y_i)|\leq\log^+|\phi(Z_i)|.
\end{align*}
Then lemma~\ref{tighter} tells us that for any 3-cell $X\in\Omega^3_-(\vec{\epsilon}_i)$,
\begin{align*}
\log^+|\phi(X)|&\geq (\log|\phi(X_i)\phi(Y_i)|-\log(4))F_{\epsilon_i}(X),\text{ and hence }\\
&\geq \frac{\log|\phi(X_i)\phi(Y_i)|-\log(4)}{\max\{F_e(X_i),F_e(Y_i),F_e(Z_i)\}}F_e(X).
\end{align*}
Therefore, if we can show that
\begin{align*}
\inf_i\left\{\frac{\log|\phi(X_i)\phi(Y_i)|-\log(4)}{\max\{F_e(X_i),F_e(Y_i),F_e(Z_i)\}}\right\}>0
\end{align*}
then we'll have shown that $\log^+|\phi|$ has a lower Fibonacci bound over all of $\Omega^3$. And to see that this holds, first notice that by going out sufficiently far from $e$, we may effectively ignore the $\log(4)$ term. Then, because $X_i$ and $Y_i$ are in $\Omega^3_0(T)$, we see that:
\begin{align*}
\frac{\log|\phi(X_i)\phi(Y_i)|}{\max\{F_e(X_i),F_e(Y_i),F_e(Z_i)\}}\geq\frac{\kappa(F_e(X_i)+F_e(Y_i))}{\max\{F_e(X_i),F_e(Y_i),F_e(Z_i)\}}\geq\frac{\kappa}{2},
\end{align*}
thus yielding the desired Fibonacci lower bound. We complete this proof by invoking lemma~\ref{upper} for the upper Fibonacci bound.
\end{proof}

\subsection{McShane Identity}\label{sec:mcshane}

Our method for proving Theorem~\ref{th:main} follows Bowditch \cite{BowPro,BowMar}. Starting with a 4-tuple of simple closed one-sided curves $\{\alpha,\beta,{\color{red}\gamma_0},{\color{blue}\gamma_1}\}$, consider the sequence of 4-tuples produced by repeatedly flipping ${\color{red}\gamma_{2i}}$ to get ${\color{red}\gamma_{2i+2}}$ and ${\color{blue}\gamma_{2i+1}}$ to get ${\color{blue}\gamma_{2i+3}}$. 
\begin{figure}[ht]
\begin{center}
\includegraphics[scale=.35]{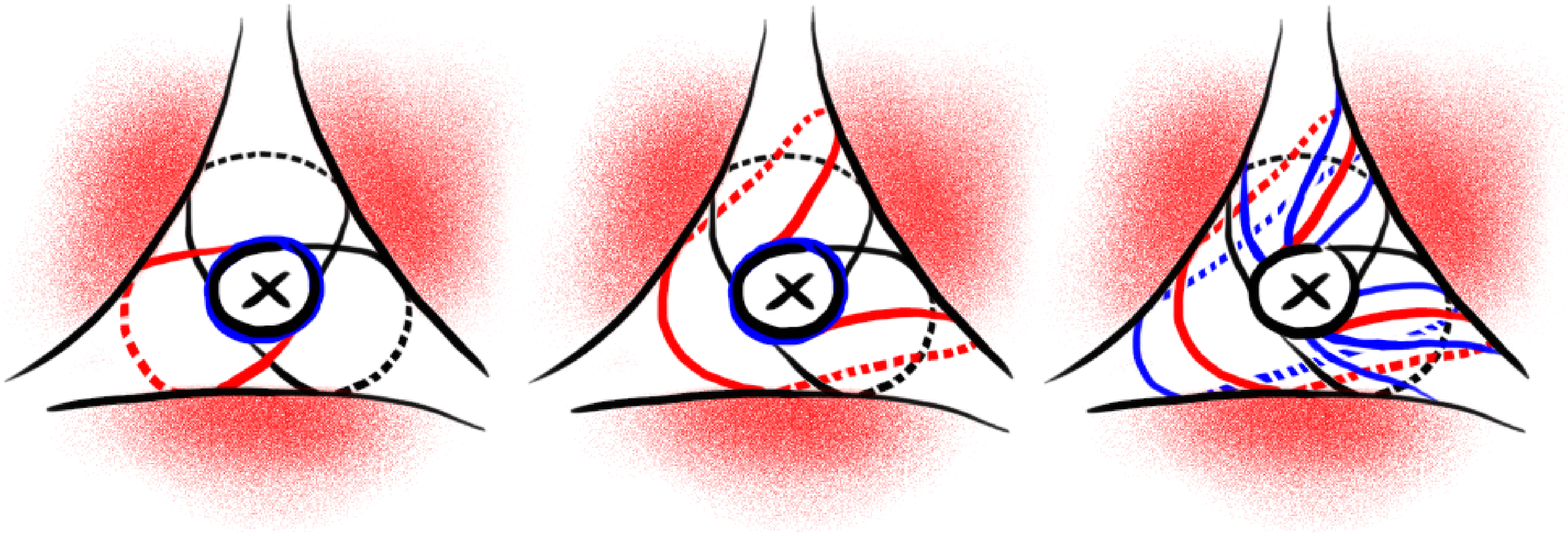}	
\end{center}
\caption{Flipping ${\color{red}\gamma_{2i}}$ followed by ${\color{blue}\gamma_{2i+1}}$.}
\label{fig:doubleflip}
\end{figure}

In this manner, we obtain a sequence of Markoff quads:
\begin{align*}
(a,b,{\color{red}c_0},{\color{blue}c_1}),(a,b,{\color{red}c_2},{\color{blue}c_1}),(a,b,{\color{red}c_2},{\color{blue}c_3}),(a,b,{\color{red}c_4},{\color{blue}c_3}),(a,b,{\color{red}c_4},{\color{blue}c_5}),\ldots
\end{align*}
By comparing the vertex relation \eqref{vertrel} at $\{\alpha,\beta,\gamma_k,\gamma_{k+1}\}$ and the edge relation \eqref{edgerel} at $\{\alpha,\beta,\gamma_k\}$, we obtain that:
\begin{align*}
\tfrac{c_{k-1}}{a+b+c_{k-1}+c_{k}}=\tfrac{c_{k}}{a+b+c_{k}+c_{k+1}}+\tfrac{a}{a+b+c_{k}+c_{k+1}}+\tfrac{b}{a+b+c_{k}+c_{k+1}}.
\end{align*}
In this decomposition for $\frac{c_{k-1}}{a+b+c_{k-1}+c_{k}}$, there is another summand of the same form but with shifted indices. Thus, starting with the the vertex relation \eqref{vertrel}, we may iteratively decompose terms of the form $\frac{c_{k-1}}{a+b+c_{k-1}+c_k}$ to obtain:
\begin{align*}
1&=\tfrac{c_0}{a+b+c_0+c_1}+\tfrac{c_1}{a+b+c_0+c_1}+\tfrac{a}{a+b+c_0+c_1}+\tfrac{b}{a+b+c_0+c_1}\\
&=\tfrac{c_1}{a+b+c_1+c_2}+\tfrac{a}{a+b+c_0+c_1}+\tfrac{b}{a+b+c_0+c_1}+\tfrac{a}{a+b+c_1+c_2}+\tfrac{b}{a+b+c_1+c_2}+\tfrac{c_1}{a+b+c_0+c_1}\\
&\hspace{5em}\vdots\\
&=\tfrac{c_n}{a+b+c_n+c_{n+1}}+\sum_{i=0}^n\left(\tfrac{a}{a+b+c_i+c_{i+1}}+\tfrac{b}{a+b+c_i+c_{i+1}}\right)+\tfrac{c_1}{a+b+c_0+c_1}.
\end{align*}
Since the edge relation \eqref{edgerel} for $\{\alpha,\beta,\gamma_n\}$ is a second order difference equation, we may explicitly compute
\begin{align*}
\lim_{n\to\infty}\frac{c_n}{a+b+c_n+c_{n+1}}=\frac{1}{1+\lambda},\text{ where }\lambda=\frac{ab-2+\sqrt{ab(ab-4)}}{2}. 
\end{align*}
When we indefinitely apply this splitting algorithm to every summand that arises, we might intuitively expect to derive a series that sums to one, whose summands each take the form $\frac{1}{1+\lambda}$ for some $\lambda$ --- this series \emph{is} the McShane identity.\newline
\newline
In the Fuchsian case, the summands $\Psi(\vec{e})$ are real numbers and correspond to the lengths of intervals in the length $1$ horocycle around one of the cusps on $X$. The complement of all of these intervals is, by construction, the union of a Cantor set and a countable set.   To show that this Cantor set is of measure $0$ in order to conclude that the series sums to $1$, we bound the measure of this Cantor set by smaller and smaller tails of the convergent series $\frac{1}{1+\lambda}$.  Convergence follows from the Fibonacci growth rates established above.\newline
\newline
Define the function $\Psi:\vec{\Omega}^1\rightarrow[0,1]$ by:
\begin{align*}
\Psi(\vec{e})=\Psi(\{\alpha,\beta,\gamma;\delta'\rightarrow \delta\}):=\frac{d}{a+b+c+d}=\frac{a+b+c+d}{abc}.
\end{align*}

Then, the edge relation \eqref{edgerel} becomes:
\begin{align*}
\Psi(\vec{e})+\Psi(\cev{e})=1,
\end{align*}
and the vertex relation \eqref{vertrel} is the following relation on four incoming oriented edges $\vec{e}_1,\vec{e}_2,\vec{e}_3,\vec{e}_4$:
\begin{align*}
\Psi(\vec{e}_1)+\Psi(\vec{e}_2)+\Psi(\vec{e}_3)+\Psi(\vec{e}_4)=1.
\end{align*}
These two properties in turn tell us that for a funnel with oriented edges $\vec{e}_1,\vec{e}_2,\vec{e}_3$ and outgoing edge $\vec{e}_4$:
\begin{align*}
\Psi(\vec{e}_4)=\Psi(\vec{e}_1)+\Psi(\vec{e}_2)+\Psi(\vec{e}_3),
\end{align*}
and so we may iteratively expand either the edge or the vertex relation \eqref{vertrel} into a statements about a finite collection of terms of the form $\Psi(\vec{e})$ summing to $1$. For a tree $T$ in the 1-skeleton of $\Omega$, if we use the notation $C(T)$ to denote
\begin{align*}
C(T):=\left\{\vec{e}\in\vec{\Omega}^1\mid\vec{e}\text{ points into, but is not contained in }T\right\},
\end{align*}
then we have:

\begin{lem}\label{equals1}
For any finite subtree $T$ in the 1-skeleton of $\Omega$,
\begin{align*}
\sum_{\vec{e}\in C(T)}\Psi(\vec{e})=1.
\end{align*}
\end{lem}

Next, define the function $h:\mathbb{C}-[0,4]\rightarrow \mathbb{C}$,
\begin{align*}
h(x)=\frac{1}{2}(1-\sqrt{1-4/x})=\frac{2}{x(1+\sqrt{1-4/x})}.
\end{align*}
For an edge $e=\{\alpha,\beta,\gamma\}$, we define
\begin{align*}
h(e)=h(\{\alpha,\beta,\gamma\}):=h\left(\tfrac{abc}{a+b+c}\right)=h\left(\left(\tfrac{1}{ab}+\tfrac{1}{ac}+\tfrac{1}{bc}\right)^{-1}\right).
\end{align*}
A little algebraic manipulation shows that:
\begin{align*}
h(e)=\Psi(\vec{e})\text{ if and only if }\mathrm{Re}(\Psi(\vec{e}))\leq\mathrm{Re}(\Psi(\cev{e})).
\end{align*}
In other words, $\Psi$ of a chosen edge $\vec{e}$ is equal to $h(e)$. In fact, the main point of Theorem~\ref{converge} is to prove the following result:

\begin{lem}\label{converge2}
The following infinite series taken over $\Omega^2$ converges absolutely for all $s>0$,
\[
\sum_{\{\xi,\eta\}\in\Omega^2}|xy|^{-s}<\infty.
\]
\end{lem}

\begin{proof}
We see from Theorem~\ref{converge} that the following series converges (absolutely):
\[
\sum_{\{\xi\}\in\Omega^3}\left|\log|x|\right|^{-3}<\infty.
\]
Hence, the following series converges:
\[
\sum_{\{\xi\}\in\Omega^3}|x|^{-\frac{s}{2}}<\infty.
\]
Squaring this series, we obtain an absolutely convergent series that's strictly greater than our desired quantity.
\end{proof}

Before we state and prove theorem~\ref{th:main}, we introduce one more piece of notation. Given a subset $E$ consisting of edges in the 1-skeleton of $\Omega$, we define:
\[
\Omega^2(E):=\left\{\{\xi,\eta\}\in \Omega^2\mid\{\xi,\eta\}\text{ contains an edge in }E\right\}.
\]

\begin{thm}\label{th:bqmcshane}
If $\phi\in\Phi_{BQ}$, then
\begin{align}
\sum_{\{\alpha,\beta\}\in\Omega^2} h(ab)=\frac{1}{2}.
\end{align}
\end{thm}

\begin{proof}
We first note that $h(x)$ is roughly order $O(|x|^{-1})$, and so lemma~\ref{converge2} tells us that:
\begin{align*}
\sum_{\{\alpha,\beta\}\in\Omega^2} h(ab)<\infty.
\end{align*}
Next, we prove an inequality of the following form:
\begin{align*}
|h(\{\alpha,\beta,\gamma\})-h(ab)|\leq \kappa|h(ac)+h(bc)|,
\end{align*}
where $\kappa>0$ is independent of $a,b$ and $c$. We begin by noting that
outside of a finite set of edges $\{\alpha,\beta,\gamma\}$, either $|a|\gg0,|b|\gg0$ or $|c|\gg0$. If $|a|$ or $|b|\gg0$, then:
\begin{align*}
\frac{|h(\{\alpha,\beta,\gamma\})-h(ab)|}{|h(ac)+h(bc)|}\approx1,
\end{align*}
and if $|c|\gg0$, then:
\begin{align*}
\frac{|h(\{\alpha,\beta,\gamma\})-h(ab)|}{|h(ac)+h(bc)|}\approx\frac{1}{2}\left|1+\sqrt{1-\frac{4}{ab}}\right|^{-1}<\frac{1}{2}.
\end{align*}
Therefore, we know that there exists a $\kappa$ satisfying our requirements.\newline
\newline
In the proof of Theorem~\ref{converge}, we construct a finite attracting tree $T$, outside of which every vertex is a funnel. Now, if we take $B_n(T)$ to be the distance $n$ neighbourhood of $T$ in the 1-skeleton of $\Omega$, then lemma~\ref{equals1} tells us that:
\begin{align*}
1=\sum_{\vec{e}\in C(B_n(T))}\Psi(\vec{e})=\sum_{\vec{e}\in C(B_n(T))} h(e).
\end{align*}
Given $\vec{e}=\{\alpha,\beta,\gamma;\delta'\rightarrow \delta\}\in C(B_n(T))$, and suppose that $\vec{e}$ joins directly onto the oriented edge $\{\alpha,\beta,\delta;\gamma\rightarrow \gamma'\}\in C(B_{n-1}(T))$, then of the three $2$-cells $\{\alpha,\beta\}$, $\{\alpha,\gamma\}$, $\{\beta,\gamma\}$ containing $e$, we know that $\{\alpha,\beta\}\in \Omega^2(B_{n-1}(T))$ and $\{\alpha,\gamma\}$, $\{\beta,\gamma\}\in \Omega^2_{n}(B_n(T))-\Omega^2(B_{n-1}(T))$. Hence, summing over all of $C(B_n(T))$, we obtain the following inequality:
\begin{align*}
\left|1-2\sum_{\{\alpha,\beta\}\in \Omega^2(B_{n-1}(T))}h(ab)\right|=&\left|\sum_{ C(B_n(T))}h(e)-2\sum_{\{\alpha,\beta\}\in \Omega^2(B_{n-1}(T))}h(ab)\right|\\
\leq&2\kappa\left|\sum_{\{\gamma,\delta\}\in \Omega^2(B_n(T))- \Omega^2(B_{n-1}(T))}h(cd)\right|,
\end{align*}
noting that we'd made use of the fact that any 2-cell meets either two or no edges in $C(B_n(T))$.
Then, by taking $n\rightarrow\infty$ and observing that the second term tends to $0$, we obtain that:
\[
1=2\sum_{\{\alpha,\beta\}\in\Omega^2}h(ab).
\]
\end{proof}

The McShane identity follows as a corollary:

\begin{repeatthm2}
Let $\rho$ be a quasi-Fuchsian representation of the thrice-punctured projective plane fundamental group $\pi_1(N_{1,3})$. Then,
\begin{align*}
\sum_{\gamma\in\Sim_2(N_{1,3})}\frac{1}{1+\exp{\tfrac{1}{2}\ell_{\gamma}(\rho)}}=\frac{1}{2},
\end{align*}
where the sum is over the collection $\Sim_2(N_{1,3})$ of free homotopy classes of essential, non-peripheral two-sided simple closed curves $\gamma$ on $N_{1,3}$.
\end{repeatthm2}

\begin{proof}
Given a simple closed two-sided geodesics $\gamma$, there is a unique pair of once-intersecting simple closed one-sided geodesics $\alpha,\beta$ that do not intersect $\gamma$ (and vice versa). Firstly, this bijection affords us the desired change in the summation indices. Secondly, by invoking the face relation \eqref{facerel}:
\begin{align*}
ab=e+2=2\cosh(\tfrac{1}{2}\ell_\gamma)+2,
\end{align*}
$h(ab)$ yields the desired summand.
\end{proof}

\begin{remark}
It is not yet clear to us whether Theorem~\ref{th:main} and Theorem~\ref{th:bqmcshane} are equivalent: if every BQ-Markoff map arises from a quasi-Fuchsian representation, then the two theorems are equivalent (and we'd have an algebraic characterisation for whether a representation is quasi-Fuchsian). If not, then Theorem~\ref{th:bqmcshane} is strictly stronger. It should be noted that it is still an open question whether Bowditch's original BQ-conditions characterise the quasi-Fuchsian punctured torus representations \cite{BowMar}.
\end{remark}

\subsection{Asymptotic growth of the simple length spectrum}  \label{sec:asymp}

A punctured Klein bottle $K$ has a unique two-sided simple closed curve $\alpha$, and a family $\alpha_i$, $i\in\mathbb{Z}$ of one-sided simple closed curves.  Set $A=2\cosh\frac{1}{2}\ell_{\alpha}(X)$ and $a_i=\sinh\frac{1}{2}\ell_{\alpha_i}(X)$ for a hyperbolic 1-cusped Klein bottle $X$.  A trace identity yields:
\begin{equation}  \label{eq:quad}
a_i^2+a_{i+1}^2-a_ia_{i+1}A=-1.
\end{equation}
hence:
\begin{align*}
\lambda_{\pm}:=\lim_{i\to\pm\infty}\frac{a_i}{a_{i+1}}\text{ satisfy }\lambda_{\pm}^2-A\lambda_{\pm}+1=0,\text{ and }\lambda_{\pm}=\exp(\pm \tfrac{1}{2}\ell_{\alpha}).
\end{align*}
 Thus, for $k\gg 0$, the sequence of traces for $\{\alpha_i\}$ is eventually approximated by:
\begin{align*}
\ldots,a_{\pm k},\ a_{\pm k}\exp\tfrac{1}{2}\ell_{\alpha},\ a_{\pm k}\exp\tfrac{2}{2}\ell_{\alpha},\ a_{\pm k}\exp\tfrac{3}{2}\ell_{\alpha},\ldots
\end{align*}
And since $2\arcsinh(\tfrac{1}{2}\cdot)\approx2\log(\cdot)$ for large numbers, the lengths for $\{\alpha_i\}$ eventually resemble:
\begin{align*}
\ldots,\log(a_{\pm k}),\ \ell_{\alpha}+\log(a_{\pm k}),\ 2\ell_{\alpha}+\log(a_{\pm k}),\ 3\ell_{\alpha}+\log(a_{\pm k}),\ldots
\end{align*}
We see therefore that $s_X(L)$ grows linearly,
\begin{align*}
s_X(L) \sim \eta(X)\cdot L=\eta(X)\cdot L^{\dim\mathcal{M}(K)}
\end{align*}
and \eqref{gropol} also holds for cusped Klein bottles. A natural question arises: \textit{does polynomial growth still hold for non-orientable surfaces?}\newline
\newline
We show using Markoff quads that the answer is no, for 1-sided simple closed geodesics.

\begin{repeatthm3} 
Given a quasi-Fuchsian representation $\rho$ of the thrice-punctured projective plane $N_{1,3}$,
\begin{align*}
\lim_{L\to\infty}\frac{s_\rho(L)}{L^m}>0
\end{align*}
for some $m$ satisfying $2.430<m < 2.477$.
\end{repeatthm3}

\begin{proof}
Let $\phi$ denote the BQ-Markoff map induced by $\rho$ and let $\phi_0$ denote the Markoff map corresponding to the $(4,4,4,4)$ Markoff quad. Theorem~\ref{converge} tells us that there is a positive number $\kappa\in\mathbb{R}^+$ such that for all but finitely many $X\in\Omega^3$,
\begin{align*}
\frac{1}{\kappa}\log^+|\phi_0(X)|\leq\log^+|\phi(X)|\leq\kappa\log^+|\phi_0(X)|.
\end{align*}
Since $2\log^+(\cdot)=2\log(\cdot)\approx2\arcsinh(\frac{1}{2}\cdot)$ for large inputs, this means that the growth rate of $\phi$ and $\phi_0$ are of the same order. The spectrum of $\phi_0$ consists of Markoff quads generated starting from $(4,4,4,4)$, and may be paraphrased in terms of the integral solutions of the $n=4$ Markoff-Hurwitz equation \ref{markoff-hurwitz}. This is known to be between the orders $L^{2.430}$ and $L^{2.477}$ by Baragar's work \cite{BarExp}, and the result follows.
\end{proof}

\bibliographystyle{plain}
\bibliography{Bibliography}

\end{document}